\documentclass[11pt]{article}

\usepackage[margin=1.0in]{geometry}
\usepackage[utf8]{inputenc}            
\usepackage[T1]{fontenc}               
\usepackage[hidelinks,backref=page]{hyperref}    	  \hypersetup{
	colorlinks,
	linkcolor={violet!50!pink},
	citecolor={violet!50!pink},
	urlcolor={violet!50!pink}
}
\usepackage[mathscr]{euscript}
\usepackage[T1]{fontenc}
\usepackage{doi}
\usepackage[version=4]{mhchem}
\usepackage{url, xurl}
\usepackage{booktabs}
\usepackage{nicefrac}
\usepackage{microtype}
\usepackage[dvipsnames,table]{xcolor} 
\usepackage{xspace}
\usepackage{comment} 

\usepackage[numbers, sort&compress]{natbib}
\usepackage{amsmath, amssymb, amsfonts, stmaryrd, dsfont}
\usepackage{amsthm, thmtools, thm-restate}
\usepackage{mathtools, nccmath, extarrows}
\usepackage[shortlabels]{enumitem} 
\usepackage{subcaption}
\usepackage{scalerel}
\usepackage{lipsum}
\usepackage{graphicx}
\usepackage{epstopdf}
\usepackage{cleveref}

\usepackage{mathrsfs, mleftright,booktabs,soul}
\usepackage{algorithm, algpseudocode, algorithmicx}

\ifpdf
  \DeclareGraphicsExtensions{.eps,.pdf,.png,.jpg}
\else
  \DeclareGraphicsExtensions{.eps}
\fi
\crefformat{footnote}{#2\footnotemark[#1]#3}

\usepackage{stackengine}
\numberwithin{equation}{section}
\stackMath

\DeclareFontEncoding{LS1}{}{}
\DeclareFontSubstitution{LS1}{stix}{m}{n}

\makeatletter
  \newcommand*\textmathversion{\csname textmv@\math@version\endcsname}
  \newcommand*\textmv@normal{m}
  \newcommand*\textmv@bold{b}
\makeatother

\usepackage{fourier}
\usepackage{bm}

\usepackage{titlesec}
\titleformat{\section}{\normalfont\scshape}{\thesection.}{1ex}{\centering}[]
\titlespacing*{\section}{0pt}{20pt}{6pt}
\titleformat{\subsection}[runin]{\normalfont\bfseries}{\thesubsection.}{6pt}{}[.]
\titlespacing{\subsection}{0pt}{6pt}{6pt}
\titleformat{\subsubsection}[runin]{\normalfont\itshape}{\thesubsubsection.}{6pt}{}[.]
\titlespacing{\subsubsection}{0pt}{6pt}{6pt}

\newcommand{\proofbox}{\vbox{\hrule height0.6pt\hbox{\vrule height1.3ex width0.6pt\hskip0.8ex\vrule width0.6pt}\hrule height0.6pt}}

\newtheorem{theorem}{Theorem}[section]
\newtheorem{proposition}[theorem]{Proposition}
\newtheorem{lemma}[theorem]{Lemma}
\newtheorem{corollary}[theorem]{Corollary}

\theoremstyle{definition}

\newtheorem{question}[theorem]{Question}

\newtheorem{property}[theorem]{Property}

\newcommand{\real}{\mathbb{R}}
\newcommand{\complex}{\mathbb{C}}

\newcommand{\onevec}{\mathbb{1}}

\DeclareMathOperator{\tr}{tr}

\newcommand{\mat}[1]{\boldsymbol{#1}}
\renewcommand{\vec}[1]{\boldsymbol{#1}}
\DeclareMathOperator{\diag}{diag}
\newcommand{\lowrank}[1]{\llbracket #1 \rrbracket}

\DeclareMathOperator{\rank}{rank}
\newcommand{\Id}{\mathbf{I}}

\DeclareMathOperator{\expect}{\mathbb{E}}
\DeclareMathOperator{\prob}{\mathbb{P}}

\newcommand{\order}{\mathcal{O}}

\DeclareMathOperator*{\argmin}{argmin}
\newcommand{\set}[1]{\mathsf{#1}}

\newcommand{\diff}{{\mathrm{d}}}

\renewcommand{\hat}[1]{\widehat{#1}}

\newcommand{\RPCholesky}{\textsc{RP\-Chol\-esky}\xspace}

\usepackage{listings}
\usepackage{color} 
\definecolor{mygreen}{RGB}{28,172,0} 
\definecolor{mylilas}{RGB}{170,55,241}
\usepackage{amsopn}

\renewcommand*{\backref}[1]{}
\renewcommand*{\backrefalt}[4]{%
	\ifcase #1 %
	(No citations.)
	\or
	(Cited on page #2.)
	\else
	(Cited on pages #2.)
	\fi
}
\newcommand*{\email}[1]{\href{mailto:#1}{\nolinkurl{#1}} }

\begin{document}

\title{Randomly pivoted Cholesky: \\ Practical approximation of a kernel matrix \\ with few entry evaluations\footnote{YC acknowledges support from the Caltech Kortschak Scholar program and the Courant Instructorship. ENE acknowledges support from the Department of Energy through Award DE-SC0021110.
JAT and RJW acknowledge support from the Office of Naval Research through BRC Award N00014-18-1-2363
and from the National Science Foundation through FRG Award 1952777.}}
\author{Yifan Chen\thanks{Courant Institute of Mathematical Sciences, New York University, NY 10012 USA (\email{yifan.chen@nyu.edu}).}
\and Ethan N. Epperly\thanks{Computing and Mathematical Sciences, California Institute of Technology, Pasadena, CA 91125 USA (\email{eepperly@caltech.edu}, \email{jtropp@caltech.edu}).} \and Joel A. Tropp\footnotemark[3] 
\and Robert J. Webber\thanks{Department of Mathematics, University of California San Diego, La Jolla, CA 92093 (\email{rwebber@ucsd.edu})}}

\date{}

\maketitle

\begin{abstract}
The randomly pivoted Cholesky algorithm (\RPCholesky) computes a factorized rank-$k$ approximation of an $N \times N$ positive-semidefinite (psd) matrix.
\RPCholesky requires only $(k + 1)N$ entry evaluations and $\order(k^2 N)$ additional arithmetic operations, and it can be implemented with just a few lines of code.  The method is particularly useful for approximating a kernel matrix.

This paper offers a thorough new investigation
of the empirical and theoretical behavior of this fundamental algorithm.
For matrix approximation problems that arise in scientific machine learning,
experiments show that \RPCholesky matches or beats the performance
of alternative algorithms.
Moreover, \RPCholesky provably returns low-rank approximations
that are nearly optimal. 
The simplicity, effectiveness, and robustness of \RPCholesky strongly support its use in scientific computing and machine learning applications.
\end{abstract}

\section{Motivation}
\label{sec:introduction}

Kernel methods~\cite{SS02} are a class of machine learning
tools for interpolation, regression, clustering, and summarization
of data.
For small to medium data sets (say, with fewer than $10^5$ points),
the literature contains evidence that kernel methods are
effective for many learning tasks in scientific
computing~\cite{SS02,rasmussen2005gaussian,RZMC11,STR+19,RSBU22,BDHO23}.
For particular problems, carefully designed kernel methods can compete
with or exceed the performance of neural networks~\cite{LSP+20,ADL+20,radhakrishnan2023mechanism,AJSW24a}.

Kernel methods distill information about the pairwise similarities
of $N$ data points into a dense positive-semidefinite
(psd) kernel matrix with dimensions $N \times N$.
We must compare two data points to determine each entry of
the kernel matrix, so it can be burdensome to compute
and store all $N^2$ entries.
To perform data analysis tasks with the kernel matrix,
we solve linear systems or least-squares problems,
or we perform eigenvalue decompositions.
With direct algorithms, these linear algebra primitives
require $\mathcal{O}(N^3)$ arithmetic.
This scaling makes it prohibitive 
to apply kernel methods to the largest data sets.

Yet the situation is not hopeless.  Even for high-dimensional
data, the eigenvalues of the kernel matrix
can decay surprisingly quickly~\cite{williams2000effect,altschuler2023kernel}.
This phenomenon has a profound consequence for computation.
\textbf{When spectral decay is present, we can replace
the full kernel matrix with a low-rank approximation
to accelerate kernel methods without much loss of accuracy.}
This approach can be used to accelerate kernel interpolation \cite{abedsoltan2023large}, kernel ridge regression \cite{rudi2017falkon,meanti2020kernel,diaz2023robust}, and kernel spectral clustering \cite{fowlkes2004spectral}.
Indeed,
``low-rank kernel methods'' can run millions (!) of times faster
than direct methods that require a full decomposition
of the dense kernel matrix (see, e.g., \Cref{sec:spectral_clustering}).

With this context in view, we pose a computational question:
\textbf{What are the best algorithms for finding a low-rank
approximation of a large psd kernel matrix?}  Here are some
desiderata:

\begin{enumerate} \setlength{\itemsep}{0pt}
\item   \textbf{Entry evaluations.}  We would
like to compute a rank-$k$ approximation after revealing only
$\mathcal{O}(kN)$ entries of the kernel matrix, such as
the $k$ most salient columns.

\item   \textbf{Arithmetic and storage.}
The method should only expend $\mathcal{O}(k^2 N)$
additional arithmetic, which is the cost to orthogonalize
$k$ vectors of dimension $N$.
The method should return the approximation in factored form,
using only $\mathcal{O}(kN)$ storage.

\item   \textbf{Approximation quality.}
The error in the computed rank-$k$ approximation
should be competitive with the best
approximation that has rank $r$, where $r$ is a number not much smaller than $k$. 

\item   \textbf{Reliability, robustness, simplicity.}
The method should have consistent performance, and
it should succeed for all inputs.  The method should
be easy to implement, and it should not require the
user to adjust parameters.
\end{enumerate}

\textbf{The randomly pivoted Cholesky algorithm (\RPCholesky)
is a fundamental numerical method that enjoys all
four of these qualities.}  \RPCholesky enhances the
classic pivoted partial Cholesky method with an adaptive,
probabilistic rule for selecting the next pivot column.
We have found that this simple modification consistently
produces excellent low-rank matrix approximations, even
when alternative pivot rules fail.
See Algorithm~\ref{alg:rpcholesky} for pseudocode.

\begin{algorithm}[t]
  \caption{\RPCholesky}
  \label{alg:rpcholesky}
  \textbf{Input:} Psd matrix $\mat{A} \in \mathbb{C}^{N\times N}$; approximation rank $k$
  
  \textbf{Output:} Pivot set $\set{S} =\{s_1, \dots, s_k\}$; matrix $\mat{F} \in \mathbb{C}^{N\times k}$ defining Nystr\"om approximation $\mat{\hat{A}} = \mat{F}\mat{F}^*$
  \begin{algorithmic}
    \State Initialize $\mat{F} \leftarrow \mat{0}_{N\times k}$ and $\vec{d} \leftarrow \diag \mat{A}$ 
    \Comment{Evaluate diagonal of input matrix}
    \For{$i = 1$ to $k$}
    \State Sample pivot $s_i \sim \vec{d} / \sum_{j=1}^N \vec{d}(j)$
    \Comment{With probability proportional to residual diagonal}
    \State $\vec{g} \leftarrow \mat{A}(:,s_i)$
    \Comment{Evaluate column $s$ of input matrix}
    \State $\vec{g} \leftarrow \vec{g} - \mat{F}(:,1:i-1) \mat{F}(s_i,1:i-1)^*$
    \Comment{Remove overlap with previously chosen columns}
    \State $\mat{F}(:,i) \leftarrow \vec{g} / \sqrt{\vec{g}(s_i)}$
    \Comment{Update approximation}
    \State $\vec{d} \leftarrow \vec{d} - |\mat{F}(:,i)|^2$
    \Comment{Track diagonal of residual matrix}
    \State $\vec{d}\leftarrow \max\{\vec{d},\vec{0}\}$ \Comment{Ensure diagonal remains nonnegative}
    \EndFor
  \end{algorithmic}
\end{algorithm}

The \RPCholesky algorithm has a subtle history
(\Cref{sec:context}), but it is fair to say that
this method has never received the attention
that it deserves. 
Our purpose is to bring this powerful algorithm into the light.
We offer two main contributions:

\begin{enumerate} \setlength{\itemsep}{0pt}
\item   \textbf{Empirical performance.}  We provide the first
numerical evidence that \RPCholesky is a competitive technique
for approximating large kernel matrices that arise in
scientific machine learning.
Compared to alternative algorithms, \RPCholesky has greater
speed, accuracy, or reliability.

\item   \textbf{Rigorous error bounds.}  We develop a new
theoretical analysis that describes the performance of the
\RPCholesky algorithm, as given in \Cref{alg:rpcholesky}.
Our analysis gives a clear picture of
why \RPCholesky is effective.
\end{enumerate}

\noindent
In summary, we present a slate of results to suggest that
\RPCholesky is the leading method for low-rank approximation
of large, psd kernel matrices.  The effectiveness, robustness,
and simplicity of \RPCholesky strongly recommend it for
modern applications in scientific computing and machine learning.

\subsection{Plan for paper}

The rest of the paper is organized as follows.
\Cref{sec:motivation} introduces \RPCholesky and its basic properties,
and \Cref{sec:context} outlines history and related work.
\Cref{sec:rpc} applies \RPCholesky to kernel ridge regression and kernel spectral clustering problems, and \Cref{sec:theoretical-analysis} establishes error bounds. \Cref{sec:conclusion} offers some conclusions.

\subsection{Notation}

The elements of a matrix $\mat{A} \in \mathbb{C}^{N \times N}$ are written $[\mat{A}(i,j)]_{1 \leq i,j \leq N}$, while the submatrices of $\mat{A}$ are expressed using MATLAB notation.
For example, $\mat{A}(:,i)$ represents the $i$th column of $\mat{A}$ and $\mat{A}(\set{S},:)$ denotes the submatrix of $\mat{A}$ with rows indexed by the set $\set{S}$.
The conjugate transpose of a (rectangular) matrix $\mat{F}$ is denoted as $\mat{F}^*$, and the Moore--Penrose pseudoinverse is $\mat{F}^{\dagger}$.  We write $\mat{\Pi}_{\mat{F}}$ for the orthogonal projector onto the column span of $\bm{F}$.

The function $\lambda_i(\mat{A})$ outputs the $i$th largest eigenvalue of a psd matrix $\mat{A}$.
The symbol $\preceq$ denotes the psd order on Hermitian matrices: $\mat{H} \preceq \mat{A}$ if and only if $\mat{A} - \mat{H}$ is psd.
We say $\mat{A}$ is rank-$r$ when $\rank \mat{A} \leq r$.  The symbol
$\lowrank{\mat{A}}_r$ refers to a best rank-$r$ approximation of a psd matrix $\mat{A}$, which can be obtained from an $r$-truncated eigendecomposition.
A best rank-$r$ approximation may not be unique, so we employ this notation
only in contexts where it leads to an unambiguous interpretation.

\section{Randomly Pivoted Cholesky} \label{sec:motivation}

In large-scale kernel methods, it is expensive to evaluate
all the entries of the psd kernel matrix.
As a cheaper alternative, we can try to approximate the
kernel matrix by adaptively evaluating a small number of
the columns.
In \Cref{sec:column_psd}, we introduce the \textit{column Nystr\"om approximation}, which is optimal among all approximations using a given set of columns.
In \Cref{sec:partial_cholesky}, we describe the \emph{pivoted partial Cholesky algorithm} as an efficient strategy for forming a column Nystr\"om approximation.
In \Cref{sec:pivot-rules}, we design a rule for selecting the columns (aka ``pivots'') in the pivoted partial Cholesky algorithm that leads to the \RPCholesky algorithm.
In \Cref{sec:compare}, we summarize simple numerical experiments 
with \RPCholesky.
Last, in \Cref{sec:theoretical}, we present a new error bound for \RPCholesky that explains why the method is effective.

\subsection{Nystr\"om approximation of a psd matrix} \label{sec:column_psd}

Let $\mat{A} \in \complex^{N \times N}$ be an arbitrary psd matrix
(not necessarily arising from a kernel computation).  To approximate the
matrix using a given subset $\set{S}$ of the column indices, we can employ the \emph{column Nystr\"om approximation} \cite[\S19.2]{MT20}:
\begin{equation}
\label{eq:nystrom}
\mat{\hat{A}} \coloneqq 
\mat{A}(:, \set{S}) \mat{A}(\set{S}, \set{S})^{\dagger} \mat{A}(\set{S}, :)
\quad\text{where}\quad \set{S} \subseteq \{1,\ldots,N\}.
\end{equation}
Since psd matrices are self-adjoint, we note that $\mat{A}(\set{S}, :) = \mat{A}(:, \set{S})^*$.
When the set $\set{S}$ contains $k$ column indices, the Nystr\"om approximation $\mat{\hat{A}}$ yields a rank-$k$ psd approximation with the following desirable properties:
\begin{enumerate}\setlength{\itemsep}{0pt}
\item   The Nystr{\"o}m approximation $\mat{\hat{A}}$ agrees with the target matrix $\mat{A}$ in the distinguished columns.  That is, $\mat{\hat{A}}(:, \set{S}) = \mat{A}(:,\set{S})$.

\item   The range of the Nystr{\"o}m approximation $\mat{\hat{A}}$ coincides with the span of the distinguished columns: $\operatorname{range}(\mat{\hat{A}}) = \operatorname{range}(\mat{A}(:,\set{S}))$.

\item   With respect to the psd order, the Nystr{\"o}m approximation $\mat{\hat{A}}$ satisfies the bounds $\mat{0} \preceq \mat{\hat{A}} \preceq \mat{A}$. 
\end{enumerate}

\noindent
In fact, the Nystr{\"o}m approximation is the maximal psd matrix that satisfies properties (2) and (3).  See~\cite[Thm.~5.3]{And05:Schur-Complements} for a rigorous statement.

We measure the quality of a Nystr{\"o}m approximation $\mat{\hat{A}}$
using the \textit{trace-norm error}:
\begin{equation} \label{eqn:trace-error}
\tr( \mat{A} - \mat{\hat{A}} ).
\end{equation}
Since $\mat{\hat{A}} \preceq \mat{A}$,
the trace-norm error is always nonnegative.
Other norms are possible, but the trace-norm error is especially meaningful in the kernel learning context \cite[\S 5.2.4]{JMLR:v24:22-0937}.
Our goal is to find a set $\set{S}$ of $k$ column indices
that makes the trace-norm error as small as possible.

\subsection{The pivoted partial Cholesky algorithm} \label{sec:partial_cholesky}

We can efficiently compute a Nystr\"om approximation \eqref{eq:nystrom} via the \emph{pivoted partial Cholesky algorithm}, presented as \Cref{alg:cholesky}.
\begin{algorithm}[t]
  \caption{Pivoted partial Cholesky algorithm}
  \label{alg:cholesky}
  \textbf{Input:} Psd matrix $\mat{A} \in \mathbb{C}^{N\times N}$; approximation rank $k$
  
  \textbf{Output:} Pivot set $\set{S} = \{s_1, \dots, s_k\}$ and matrix $\mat{F} \in \mathbb{C}^{N\times k}$ defining Nystr\"om approximation $\mat{\hat{A}} = \mat{F}\mat{F}^*$
  \begin{algorithmic}
    \State Initialize $\mat{F} \leftarrow \mat{0}_{N\times k}$
    \For{$i = 1$ to $k$}
    \State Select a pivot $s_i \in \{1, \ldots, N\}$
    \Comment{See \Cref{sec:pivot-rules} for pivot rules}
    \State $\vec{g} \leftarrow \mat{A}(:,s_i)$
    \Comment{Evaluate the $s_i$ column of the input matrix $\mat{A}$}
    \State $\vec{g} \leftarrow \vec{g} - \mat{F}(:,1:(i-1)) \mat{F}(s_i,1:(i-1))^*$
    \Comment{Remove the influence of previously chosen columns}
    \State $\mat{F}(:,i) \leftarrow \vec{g} / \sqrt{\vec{g}(s_i)}$
    \EndFor
  \end{algorithmic}
\end{algorithm}

Conceptually, the algorithm begins with an initial approximation $\mat{\hat{A}}^{(0)} = \mat{0}$ and an initial residual $\mat{A}^{(0)} = \mat{A}$.
At each step $i= 1,2, 3, \dots, k$, we adaptively select a new column index
$s_i \in \{1, \dots, N\}$, called a \textit{pivot}, using some pivot rule (\Cref{sec:pivot-rules}).
Then we evaluate the $s_i$ column $\mat{A}^{(i-1)}(:, s_i)$
of the current residual, and we use this column to update
the approximation and the residual:
\begin{equation} \label{eqn:chol-updates}
\begin{aligned}
\mat{\hat{A}}^{(i)} &= \mat{\hat{A}}^{(i-1)} + \frac{\mat{A}^{(i-1)}(:,s_i)\mat{A}^{(i-1)}(s_i,:)}{\mat{A}^{(i-1)}(s_i, s_i)}; \\
{\mat{A}}^{(i)} &= {\mat{A}}^{(i-1)} - \frac{\mat{A}^{(i-1)}(:,s_i)\mat{A}^{(i-1)}(s_i,:)}{\mat{A}^{(i-1)}(s_i, s_i)}.
\end{aligned}
\end{equation}
We can also track the diagonal of the residual $\mat{A}^{(i)}$ using
the formula
\begin{equation} \label{eqn:cholesky-diag}
\diag( \mat{A}^{(i)} ) = \diag( \mat{A}^{(i-1)} ) - \frac{1}{\mat{A}^{(i-1)}(s_i, s_i)} \cdot \vert \mat{A}^{(i-1)}(:, s_i) \vert^2.
\end{equation}
The function $\vert \cdot \vert^2 : \complex^N \to \real_+^N$
returns the entrywise squared magnitude of a vector.
This observation supports stopping rules based on functions
of the diagonal entries of the residual, such as its trace.

The practical implementation of pivoted partial Cholesky (\Cref{alg:cholesky})
maintains the approximation in factored form:
$\mat{\hat{A}}^{(i)} = \mat{F}(:, 1:i) \mat{F}(:,1:i)^*$
where $\mat{F} \in \complex^{N \times k}$.  We generate
the columns of the factor $\mat{F}$ sequentially.
The $i$th column $\mat{F}(:, i)$ is obtained
by evaluating the $s_i$ column $\mat{A}(:, s_i)$
of the input matrix and removing the influence of
the previously selected columns.

The following classic result~\cite[p.~24]{HZ05:Basic-Properties}
connects the pivoted partial Cholesky algorithm with the
column Nystr{\"o}m approximation.

\begin{property}[Pivoted partial Cholesky computes a Nystr{\"o}m approximation]
    The pivoted partial Cholesky algorithm (\Cref{alg:cholesky}) with psd input matrix $\mat{A}$ and with pivot set $\set{S} = \{s_1,\ldots,s_k\}$ returns the column Nystr\"om approximation $\mat{\hat{A}} = \mat{A}(:, \set{S}) \mat{A}(\set{S}, \set{S})^{\dagger} \mat{A}(\set{S}, :)$ in the factorized form $\mat{\hat{A}} = \mat{FF}^*$.
\end{property}

\noindent
Indeed, at each step of the pivoted partial Cholesky algorithm, $\mat{\hat{A}}^{(i)}$ is the Nystr{\"o}m approximation
of $\mat{A}$ using the columns indexed by $\{s_1, \dots, s_i\}$.

What are the computational costs?  For $k$ steps of the pivoted partial Cholesky algorithm, we evaluate $kN$ entries of the input matrix $\mat{A}$, plus any other entries required to implement the pivot rule.  We expend $\mathcal{O}(k^2 N)$ additional
arithmetic operations, and the algorithm needs $\mathcal{O}(kN)$ storage.

In practice, we typically run the pivoted partial Cholesky method until the trace-norm error
falls below a specified threshold:
$\tr (\mat{A} - \mat{\hat{A}} ) \leq \eta \cdot \tr \mat{A}$.
This modification requires an unpredictable number of steps,
but it controls the error level and ensures
that $\mat{\hat{A}}$ can be reliably used in place
of $\mat{A}$ for downstream computations.

\subsection{Pivot selection rules}
\label{sec:pivot-rules}

In the pivoted partial Cholesky algorithm, we select a new pivot at each step.
Although there are several natural strategies, it has long remained unclear
how to quickly and reliably select pivots that result in matrix approximations
that control the trace-norm error~\eqref{eqn:trace-error}.

We have already seen that we can track the diagonal of the residual matrix via~\eqref{eqn:cholesky-diag}.
One principled approach for pivot selection is to exploit the information
contained in the diagonal.  Indeed, the diagonal entries of a psd matrix $\mat{A} \in \mathbb{C}^{N \times N}$ are nonnegative, and they control the off-diagonal entries via the inequality
\begin{equation*}
    |\mat{A}(i,j)| \leq \sqrt{\mat{A}(i,i) \mat{A}(j,j)} \quad \text{for each } 1 \leq i,j \leq N.
\end{equation*}
Thus, a large diagonal entry $\mat{A}(j,j)$ indicates that column $j$ might contain large-magnitude entries.
By eliminating such a column, we can hope to substantially reduce the trace of the residual at each step.  We outline several established strategies based on this intuition.

\subsubsection{Greedy pivoting}

Because of the significance of large diagonal entries,
we might be tempted to use a \emph{greedy} pivoting strategy.
At step $i$ of the pivoted partial Cholesky method, the greedy method selects the pivot $s_{i}$ by finding the position of the largest diagonal entry of the residual matrix $\mat{A}^{(i-1)}$,
with ties broken arbitrarily.
In symbols, the greedy pivot rule is
\begin{equation}
\label{eq:greedy}
    s_{i} \in \argmin\nolimits_{1 \leq j \leq N}\ \mat{A}^{(i-1)}(j, j).
\end{equation}
This greedy pivoting strategy has long been used in scientific computing and kernel machine learning, under the name ``Cholesky with complete pivoting'' \cite{high90c}.
The strategy is entirely based on \emph{exploiting} the large diagonal entries without \emph{exploring} any smaller ones.
The overemphasis on exploitation makes the greedy method brittle,
as this 
algorithm is often derailed by the presence of outlier columns.

\subsubsection{Uniform random pivoting}

The opposite strategy from greedy pivoting
is to sample each pivot
uniformly at random:
\begin{equation*}
    s_{i} \sim \textsc{uniform}\{1, \ldots, N\} \qquad \text{for each } i = 1, \ldots, k.
\end{equation*}
If we select a pivot that we have already seen, we can draw a sample again.
Uniform sampling has been popular in kernel machine learning since the works of Williams \& Seeger \cite{WS00}
and Drineas \& Mahoney \cite{DM05}.
Uniform sampling has the reverse problem from the greedy method: it randomly \textit{explores} without \textit{exploiting} any information from the diagonal.
As a result, uniform sampling leads to a poor approximation when there are small sets of columns that are distinct from all the others.

\subsubsection{Adaptive random pivoting}

This paper advocates for a third way.  The randomly pivoted Cholesky (\RPCholesky) algorithm balances \emph{exploration} of the columns with \emph{exploitation} of the information available from the diagonal.  At the $i$th step, \RPCholesky adaptively samples the pivot $s_{i}$ according to the probability distribution that is proportional to the diagonal entries of the current residual $\mat{A}^{(i-1)}$:
    \begin{equation}
        \label{eq:sampling_dist}
        \mathbb{P}\{s_{i} = j\} = \frac{\mat{A}^{(i-1)}(j, j)}{\tr \mat{A}^{(i-1)}}
        \quad \text{for } j = 1, \ldots, N.
    \end{equation}
\Cref{alg:rpcholesky} contains a simple implementation of \RPCholesky.
See \Cref{sec:context} for the history and some related algorithms.

\subsubsection{Pivoting with a Gibbs distribution}

In retrospect, we realize that the greedy method, uniform sampling, and \RPCholesky are connected.
Consider the pivot rule that selects the next pivot $s_{i}$ from a Gibbs probability distribution that is proportional to the diagonal entries of the residual matrix $\mat{A}^{(i-1)}$, after raising them to a power $\beta \in [0, \infty]$:
\begin{equation}
\label{eq:parametric_family}
    \mathbb{P}\{s_{i} = j\} = \frac{\mat{A}^{(i-1)}(j,j)^\beta}{\sum_{k=1}^N \mat{A}^{(i-1)}(k,k)^\beta}
    \quad \text{for each $j = 1, \ldots, N$.}
\end{equation}
The greedy method arises from the $\beta \rightarrow \infty$ limit (``zero temperature''), while uniform sampling arises from the $\beta \rightarrow 0$ limit (``infinite temperature'').  \RPCholesky takes the intermediate value $\beta = 1$ (``not too hot, not too cold'').
Our analysis (\Cref{sec:theoretical-analysis})
proves that \RPCholesky ($\beta=1$) yields rigorous error bounds,
while we provide examples (\Cref{sec:related-work}) where
the extreme strategies ($\beta \to 0$ and $\beta \to \infty$) fail.
The literature contains empirical evidence that values $\beta \notin \{ 0,1,\infty \}$ can be effective for some matrices \cite{Ste24}.

\subsection{Numerical results: Illustrative examples}
\label{sec:compare}

\begin{figure}[t]
    \centering
    \includegraphics[width=0.98\textwidth]{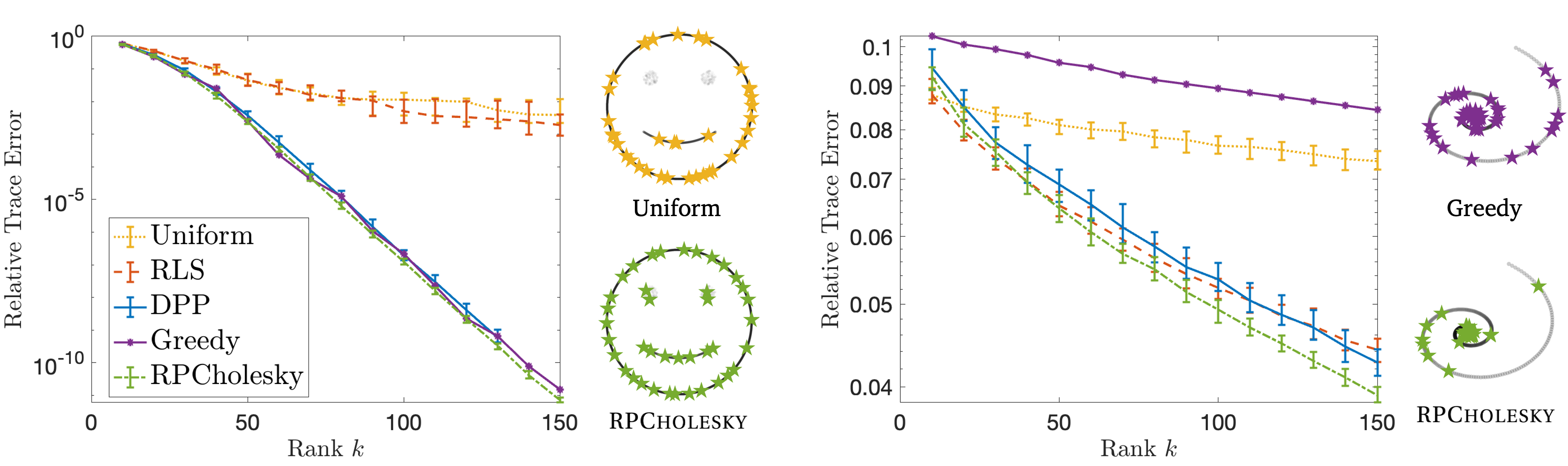}
    
    \caption{\textbf{Rank-$k$ approximation of Gaussian kernel matrices.} Median relative trace-norm error $\tr \bigl(\mat{A} - \mat{\hat{A}}^{(k)}\bigr) \slash \tr \mat{A}$ and 20-80\% quantile bars for several Nystr\"om-based column approximation methods for \textbf{Smile} (\emph{left}) and \textbf{Spiral} (\emph{right}) examples.
    Selected pivots (colored stars) and data points (gray circles) for uniform, greedy, and \RPCholesky methods are shown next to each panel.
    \label{fig:comparison}}
\end{figure}

In this section, we present two stylized examples to highlight the benefits of
the \RPCholesky method and the potential weaknesses of some other matrix approximation
algorithms.
Real-world examples appear in \Cref{sec:testbed,sec:rpc}.

We applied several column Nystr\"om approximation schemes to two kernel matrices:
\begin{enumerate}
    \item \textbf{Smile:} A Gaussian kernel matrix constructed from $10^4$ data points depicting a smile in $\real^2$.
    The smile is located in $[-10, 10] \times [-10, 10]$ and the kernel bandwidth is $\sigma = 2.0$.
    The eyes are constructed from $10^2$ points, making them easy to miss for certain sampling methods.
    \item \textbf{Spiral:} A Gaussian kernel matrix constructed from $10^4$ data points in $\real^2$ depicting the logarithmic spiral $({\rm e}^{0.2t}\cos t,{\rm e}^{0.2t}\sin t)$ for non-equispaced parameter values $t \in [0, 64]$.
    The kernel bandwidth is $1000$, making the outer edge of the spiral a region of outliers in the data.
\end{enumerate}
For each data set, \Cref{fig:comparison} tracks the mean relative trace-norm error $\tr(\mat{A}-\mat{\hat{A}})/\tr(\mat{A})$ over $100$ independent trials.
The pivots selected by different Nystr\"om-based column selection methods with $k = 40$ are shown next to each panel.
See \Cref{app:numerical_details} for additional computational details.

These tests bring to light the failure modes of uniform sampling and the greedy method.
Uniform sampling fails to select the pivots representing less populated regions of the data space,
such as the eyes in the \textbf{Smile} example.
The greedy method heavily emphasizes outliers,
leading to poor approximation accuracy in the \textbf{Spiral} example.
By contrast, \RPCholesky avoids these failure modes and achieves high accuracy for both test problems.

\Cref{fig:comparison} also evaluates two other randomized column selection schemes for Nystr\"om approximation: \emph{determinantal point process} (DPP) sampling \cite{DM21a} and \emph{ridge leverage score} (RLS) sampling \cite{MM17,AM15,RCCR18} (using the RRLS algorithm \cite{MM17}).
These methods offer strong theoretical guarantees, yet existing samplers are complicated and 
require care in implementation to avoid failure.
In practice, they also require far more entry evaluations than \RPCholesky to achieve the same approximation quality.

    

Computational cost is another important difference between various Nystr\"om column selection methods.
One simple measure of cost is the total number of entry evaluations required to generate the pivots and then form the Nystr\"om approximation.
Uniform sampling is the cheapest Nystr\"om method, requiring just $kN$ entry evaluations.
The greedy method and \RPCholesky follow closely behind, requiring $(k+1)N$ entry evaluations.
For the examples in \Cref{fig:comparison}, RLS sampling requires roughly $3kN$ entry evaluations and DPP sampling (using the \texttt{vfx} sampler \cite{GPBV19}) requires between $80kN$ and $200kN$ entry evaluations,
making these latter two methods comparatively expensive.

RLS sampling and DPP sampling exhibit other performance issues in addition to the high cost.
RLS sampling often fails to provide a high-quality approximation for the $\textbf{Smile}$ kernel matrix.
With an approximation rank of $k = 40$, there is a $95\%$ chance of missing both eyes, which is better than the $99.6\%$ chance of missing both eyes with uniform sampling, but still indicates a failure mode for the RLS algorithm.
DPP sampling using the \texttt{vfx} algorithm
\cite{GPBV19}
often fails to produce any output, generating an error that the matrix is close to rank-deficient.
With existing software, it is necessary to switch to a slower
\texttt{GS} sampler \cite{GPBV19} based on a complete eigendecomposition of the matrix $\mat{A}$.
Even the \texttt{GS} sampler fails for the \textbf{Smile} kernel matrix for $k \geq 140$.

\subsection{Numerical results: Real data} \label{sec:testbed}

\begin{table}[t]
    \centering

\begin{tabular}{rcccccc}
\toprule
                                             & Unif    & RLS     & Greedy  & \textsc{B-RPChol} & \textsc{RPChol} & Opt     \\ \midrule
\texttt{sensit\_vehicle}        & 1.57e-1 & 1.40e-1 & 2.07e-1 & \textbf{1.37e-1}               & \textbf{1.37e-1}             & 8.77e-2 \\
\texttt{yolanda}                & 1.46e-1 & 1.39e-1 & 2.08e-1 & 1.37e-1                        & \textbf{1.36e-1}             & 8.41e-2 \\
\texttt{YearPredictionMSD}      & 1.30e-1 & 1.20e-1 & 1.73e-1 & 1.18e-1                        & \textbf{1.16e-1}             & 6.79e-2 \\
\texttt{w8a}                    & 1.42e-1 & 1.17e-1 & 1.91e-1 & 1.14e-1                        & \textbf{1.05e-1}             & 6.09e-2 \\
\texttt{MNIST}                  & 1.21e-1 & 1.10e-1 & 1.67e-1 & 1.07e-1                        & \textbf{1.06e-1}             & 5.83e-2 \\
\texttt{jannis}                 & 1.11e-1 & 1.10e-1 & 1.28e-1 & \textbf{1.09e-1}               & \textbf{1.09e-1}             & 5.45e-2 \\
\texttt{HIGGS}                  & 6.73e-2 & 6.31e-2 & 8.51e-2 & 6.14e-2                        & \textbf{6.07e-2}             & 2.96e-2 \\
\texttt{connect\_4}             & 5.81e-2 & 5.07e-2 & 6.48e-2 & 4.85e-2                        & \textbf{4.81e-2}             & 2.25e-2 \\
\texttt{volkert}                & 5.35e-2 & 4.42e-2 & 5.92e-2 & 4.26e-2                        & \textbf{4.17e-2}             & 1.99e-2 \\
\texttt{creditcard}             & 4.83e-2 & 3.71e-2 & 5.77e-2 & 3.26e-2                        & \textbf{3.01e-2}             & 1.31e-2 \\
\texttt{Medical\_Appointment}   & 1.74e-2 & 1.43e-2 & 1.92e-2 & 1.31e-2                        & \textbf{1.29e-2}             & 4.59e-3 \\
\texttt{sensorless}             & 1.20e-2 & 7.78e-3 & 8.70e-3 & 8.23e-3                        & \textbf{5.80e-3}             & 2.11e-3 \\
\texttt{ACSIncome}              & 9.93e-3 & 5.55e-3 & 8.35e-3 & 4.26e-3                        & \textbf{4.02e-3}             & 1.27e-3 \\
 \texttt{Airlines\_DepDelay\_1M} & 4.19e-3 & 2.37e-3 & 2.64e-3 & 1.93-3                         & \textbf{1.78e-3}             & 5.08e-4 \\
\texttt{covtype\_binary}        & 9.10e-3 & 2.12e-3 & 1.41e-3 & 1.26e-3                        & \textbf{1.04e-3}             & 2.97e-4 \\
\texttt{diamonds}               & 1.31e-3 & 2.40e-4 & 1.12e-4 & 1.70e-4                        & \textbf{5.85e-5}             & 1.30e-5 \\
\texttt{hls4ml\_lhc\_jets\_hlf} & 3.78e-4 & 7.30e-5 & 6.68e-5 & 5.36e-5                        & \textbf{4.38e-5}             & 1.04e-5 \\
\texttt{ijcnn1}                 & 3.91e-5 & 3.09e-5 & 2.86e-5 & 2.62e-5                        & \textbf{2.13e-5}             & 4.67e-6 \\
\texttt{cod\_rna}               & 6.00e-4 & 1.36e-5 & 9.19e-6 & 1.37e-5                        & \textbf{5.05e-6}             & 9.86e-7 \\
\texttt{COMET\_MC\_SAMPLE}      & 3.65e-3 & 2.44e-7 & 1.2e-10 & 6.43e-5                        & \textbf{4.3e-11}             & 3.5e-12 \\ \bottomrule
\end{tabular}
    \caption{\textbf{Comparison of Nystr\"om methods.} Relative trace error of rank-$1000$ Nystr\"om approximation of kernel matrices produced by the RLS, uniform, \RPCholesky, greedy, and block \RPCholesky methods on a testbed of examples.
    The error of the optimal rank-$1000$ approximation is shown for reference, and we report the median error of ten trials.
    \RPCholesky achieves the smallest approximation error for every test problem.}
    \label{tab:many-matrices}
\end{table}

To confirm that the findings of \Cref{sec:compare} extend to real data sets, we compare different Nystr\"om methods on a testbed of kernel matrices formed from twenty data sets in the LIBSVM \cite{chang2011libsvm}, OpenML \cite{vanschoren2014openml}, and UCI \cite{Dua:2019} repositories.
These examples are cataloged in \cite[Tab.~1]{diaz2023robust}.
We subsample each data set to 
$N=10^4$ points, standardize each feature, and form the Gaussian kernel matrix with bandwidth $\sigma = \sqrt{d}$, where $d$ is the number of features.
(See \Cref{sec:kernel_basics} for a primer on kernel matrices.)
We omit DPP sampling due to its high computational cost, and we include the block \RPCholesky method (introduced below in \Cref{sec:block}) with block size $T=100$.
For each method, we report the median relative trace error over ten trials.

Results are shown in \Cref{tab:many-matrices}.
\RPCholesky achieves the lowest trace error on every problem in the testbed.
It achieves $8\times 10^7$ and $5\times 10^3$ times lower error than uniform and RLS sampling for the data set \texttt{COMET\_MC\_SAMPLE} with the fastest spectral decay.
\RPCholesky achieves a more modest $3\times$ improvement over the greedy method in this example.

\subsection{Theoretical results} \label{sec:theoretical}

Given the excellent empirical performance of the \RPCholesky algorithm, it is natural to seek a more rigorous explanation.
We present the first proof that \RPCholesky (\Cref{alg:rpcholesky}) with the standard parameter choices attains error bounds that are nearly optimal within the class of column Nystr{\"o}m approximations.  See \Cref{sec:rpqr-theory} and \Cref{sec:related-work} for prior theoretical work.

Let $\mat{A}$ be a psd target matrix, and let $\mat{\hat{A}}$ be a rank-$k$ column Nystr{\"o}m approximation of the form~\eqref{eq:nystrom}.
Then it is appropriate to compare the approximation error, $\tr(\mat{A} - \mat{\hat{A}})$,
against the error $\tr(\mat{A} - \lowrank{\mat{A}}_r)$ attained by a best rank-$r$ psd approximation $\lowrank{\mat{A}}_r$ where the parameter $r \leq k$.

A Nystr{\"o}m approximation $\mat{\hat{A}}$ using $k$ randomly chosen columns
is called an \textit{$(r, \varepsilon)$-approximation} of the target matrix $\mat{A}$ when
\begin{equation} \label{eq:1+eps}
\expect \tr( \mat{A} - \mat{\hat{A}} ) \leq (1 + \varepsilon) \cdot \tr( \mat{A} - \lowrank{\mat{A}}_r )
\end{equation}
for parameters $r \in \mathbb{N}$ and $\varepsilon > 0$.
The expectation averages over the random choice of columns (e.g., the random pivots in \RPCholesky).
Our theory addresses the following question:

\begin{question}
How many columns $k$ are sufficient to guarantee that a randomized column Nystr{\"o}m method attains
an $(r, \varepsilon)$-approximation~\eqref{eq:1+eps} for every $N \times N$ psd input matrix?
\end{question}

To achieve an $(r, \varepsilon)$-approximation for a worst-case matrix,
a column Nystr{\"o}m approximation must use at least $k \geq r \slash \varepsilon$
columns.  For example, see \Cref{thm:lower-bound}.
We will establish that \RPCholesky achieves an $(r,\varepsilon)$-approximation
for every psd matrix after accessing only a little more than $r \slash \varepsilon$ columns.  Here is a partial statement of our main result (\Cref{thm:main_bound}):

\begin{theorem}[Randomly pivoted Cholesky: simplified bound] \label{thm:simple_bound}
Fix $r \in \mathbb{N}$ and 
$\varepsilon > 0$, and let $\mat{A}$ be a psd matrix.
The column Nystr\"om approximation $\mat{\hat{A}}^{(k)}$ produced by \RPCholesky (\Cref{alg:rpcholesky})
attains the error bound~\eqref{eq:1+eps}
provided that the number of columns, $k$, satisfies 
\begin{equation} \label{eq:rpc-bd-intro}
    k \geq \frac{r}{\varepsilon} + r \log\biggl(\frac{1}{\varepsilon \eta}\biggr)
    \quad\text{where}\quad
    \eta \coloneqq \frac{\tr(\mat{A} - \lowrank{\mat{A}}_r)}{\tr(\mat{A})}.
\end{equation}
\end{theorem}

\begin{table}[t]
  \centering
  \renewcommand{\arraystretch}{1.1}
  \begin{tabular}{ccc}
    \toprule
    Method & Number $k$ of columns to achieve \eqref{eq:1+eps} & Reference\\ \midrule
    Greedy method & $(1 - (1 + \varepsilon)\eta) N$ & \Cref{thm:greedy} \\
    Uniform sampling* & $\frac{r-1}{\varepsilon \eta} + \frac{1}{\varepsilon}$ & \cite{FKV04}, \Cref{thm:diagonal} \\
    DPP sampling & $\frac{r}{\varepsilon} + r - 1$ &
    \cite{BW09,GS12}, \Cref{thm:near-optimal}
    \\
    RLS sampling${}^\dagger$ & $65 \bigl(\frac{r}{\varepsilon} + r\bigr) \log\bigl(\frac{4}{\eta} \bigl(\frac{r}{\varepsilon} + r\bigr)\bigr)$ & \Cref{thm:musco} \\
    \midrule
    \RPCholesky & $\frac{r}{\varepsilon} + r \log \bigl(\frac{1}{\varepsilon \eta}\bigr)$ 
    & \Cref{thm:main_bound}  \\
    \RPCholesky & $\frac{r}{\varepsilon} + r + r \log_+ \bigl(\frac{2^r}{\varepsilon}\bigr)$ 
    & \Cref{thm:main_bound} \\
    \midrule
    Lower bound & $\frac{r}{\varepsilon}$ & \cite{DV06,GS12}, \Cref{thm:lower-bound}
 \\\bottomrule
  \end{tabular}
  \caption{\textbf{Bounds for column Nystr{\"o}m approximations.}  Upper bounds and a lower bound on the number of columns that several Nystr{\"o}m approximation schemes use to produce an $(r,\varepsilon)$-approximation \eqref{eq:1+eps}.  The parameter $\eta$ is the relative error (see \Cref{thm:simple_bound}).  All logarithms have base $\mathrm{e}$, and $\log_+(x) \coloneqq \max\{\log x, 0\}$ for $x > 0$.
  *The result for uniform sampling assumes the diagonal entries of $\mat{A}$ are all equal.
  ${}^\dagger$See \Cref{sec:rls} for discussion.}
  \label{tab:comparison}
\end{table}

For comparison, \Cref{tab:comparison} presents the best available upper bounds on the number $k$ of columns for \RPCholesky and other column Nystr{\"o}m approximation methods
to achieve~\eqref{eq:1+eps}.
These bounds are expressed in terms of $r$, $\varepsilon$, $N$, and the relative approximation error $\eta$ defined in \eqref{eq:rpc-bd-intro}.
DPP sampling satisfies the strongest error bounds of all the methods in \Cref{tab:comparison}.  To achieve an $(r, \varepsilon)$-approximation with DPP sampling,
it suffices to take the number $k$ of columns as
\begin{equation} \label{eq:cannot_improve}
    k \geq \frac{r}{\varepsilon} + r - 1.
\end{equation}
\RPCholesky satisfies the second strongest error bound~\eqref{eq:rpc-bd-intro}.
The main difference between the DPP result \eqref{eq:cannot_improve} and the \RPCholesky result \eqref{eq:rpc-bd-intro}
is the multiplicative factor $\log(1/\eta)$ present in the latter.
However, because the relative error $\eta$ appears inside the logarithm, this factor has only a modest impact on the computational scaling.
Indeed, $\log (1/\eta)$ is between 2.4 and 26 for all twenty examples in \Cref{tab:many-matrices}.
We establish \Cref{thm:simple_bound} in \Cref{sec:theoretical-analysis},
which contains additional results and discussion.

The bound on the approximation rank for \RPCholesky is at least $65\times$ smaller than the bound for RLS sampling.
However, this quantitative dependence may be a result of the proof technique that overstates the difference between the methods.
See \Cref{sec:rls} for the proof of the RLS error bounds and related discussion.

The error bounds for \RPCholesky are also \emph{significantly stronger} than the bounds for the greedy method and uniform sampling.
For a worst-case matrix $\mat{A}$, the greedy method requires $\Theta(N)$ columns to approach the best rank-$r$ approximation error (\Cref{sec:greedy}),
while the uniform sampling method requires $\Theta(r/\eta)$ columns (\Cref{sec:uniform}).
In contrast, \RPCholesky uses a number of columns that is independent of the dimension $N$ and depends only logarithmically on the relative error $\eta$.
These results help explain why \RPCholesky does not exhibit the same failure modes as the greedy method and uniform sampling.

\section{History, Related Work, and Extensions} \label{sec:context}

To understand the history of the \RPCholesky algorithm,
we must reinterpret it as a randomly pivoted QR algorithm.
Indeed, almost all of the existing theoretical and numerical work
that is relevant to \RPCholesky is framed in terms of
randomly pivoted QR algorithms, which---unlike \RPCholesky---require reading all entries of the input matrix.  In this section, we will explore
this connection, discuss prior work, and describe
related algorithms.

\subsection{Nystr{\"o}m approximations and projection approximations}
\label{sec:close_relationship}

We begin with an alternative perspective on the
column Nystr{\"o}m approximation.
Let $\mat{A} \in \complex^{N \times N}$ be a psd matrix,
and select any factorization $\mat{A} = \mat{B}^* \mat{B}$
where the factor $\mat{B} \in \complex^{M \times N}$.
For any subset $\set{S} \subseteq \{1, \dots, N \}$ of column
indices, the Nystr{\"o}m approximation $\mat{\hat{A}} = \mat{A}(:, \set{S}) \mat{A}(\set{S}, \set{S})^{\dagger} \mat{A}(\set{S}, :)$
of the target matrix $\mat{A}$ admits the representation
\begin{equation*}
\mat{\hat{A}} = \mat{B}^* \mat{\Pi}_{\mat{B}(:, \set{S})} \mat{B},
\end{equation*}
where $\mat{\Pi}_{\mat{M}}$ denotes the orthogonal projector
onto $\operatorname{range}(\mat{M})$.  To check this claim,
set $\mat{M} = \mat{B}(:, \set{S})$ and decompose the projector
as $\mat{\Pi}_{\mat{M}} = \mat{M}(\mat{M}^* \mat{M})^{\dagger} \mat{M}^*$.

Equivalently, the Nystr{\"o}m approximation takes the form
\[
\mat{\hat{A}} = \mat{\hat{B}}^* \mat{\hat{B}}
\quad\text{where}\quad
\mat{\hat{B}} = \mat{\Pi}_{\mat{B}(:, \set{S})} \mat{B}.
\]
We call $\mat{\hat{B}}$ a \textit{column projection approximation}
of $\mat{B}$ with respect to the column index set $\set{S}$.
The trace-norm error in the Nystr{\"o}m approximation
can be expressed in terms of the projection approximation:
\[
\tr(\mat{A} - \mat{\hat{A}})
    = \Vert \mat{B} - \mat{\hat{B}} \Vert_{\mathrm{F}}^2.
\]
Therefore, the problem of finding a set $\set{S}$ of $k$ columns
to minimize the trace-norm error in the Nystr{\"o}m approximation
of $\mat{A}$ is the same as the problem of finding a set $\set{S}$ of $k$
columns to minimize the squared Frobenius-norm error in the
projection approximation of the factor $\mat{B}$.
Additionally,
the projection approximation $\mat{\hat{B}}$
is the best Frobenius-norm approximation of $\mat{B}$
with the same range as $\mat{B}(:, \set{S})$.

\subsection{Partial Cholesky and partial QR}

Just as we compute the column Nystr{\"o}m approximation by means of
the pivoted partial Cholesky algorithm, we can compute a column projection
approximation via the classical column-pivoted partial QR algorithm 
(\Cref{alg:pivoted-qr} in the appendix).

Here is a conceptual description.
For an input matrix $\mat{B}$, the column-pivoted partial QR algorithm initializes the approximation $\mat{\hat{B}}^{(0)} = \mat{0}$ and the residual $\mat{B}^{(0)} = \mat{B}$.  At each step $i$, we choose a column index $s_i$ using some pivot rule (\Cref{sec:qr-pivot-rules}).  The updated approximation $\mat{\hat{B}}^{(i)}$ is the projection approximation of $\mat{B}$ with respect to the selected columns $\{s_1, \dots, s_i\}$.  The updated residual is $\mat{B}^{(i)} = \mat{B} - \mat{\hat{B}}^{(i)}$.  We repeat for $k$ steps or until we trigger a stopping criterion.

The relationship between pivoted partial QR and pivoted partial Cholesky
is classical \cite[\S5.2]{high90c}:

\begin{property}[Pivoted partial Cholesky and pivoted partial QR] \label{prop:chol-qr}
Assume that $\mat{A} = \mat{B}^* \mat{B}$.  Suppose
pivoted partial Cholesky (\Cref{alg:cholesky}) selects pivot set $\set{S}$
and outputs an approximation $\mat{\hat{A}}$, while
partial QR (\Cref{alg:pivoted-qr}) selects the same pivot set $\set{S}$
and outputs an approximation $\mat{\hat{B}}$.  Then the approximations
satisfy $\mat{\hat{A}} = \mat{\hat{B}}^* \mat{\hat{B}}$.
\end{property}

While pivoted partial QR and pivoted partial Cholesky are algebraically related, these procedures should ideally be applied to different matrices and have different computational costs.
To construct a rank-$k$ approximation of a rectangular matrix $\mat{B} \in \complex^{M\times N}$, pivoted partial QR reads all $MN$ entries of $\mat{B}$ and expends $\order(kMN)$ additional arithmetic operations.
By contrast, pivoted partial Cholesky construct a rank-$k$ approximation of a psd matrix $\mat{A} \in \complex^{N \times N}$ by looking at just $kN$ matrix entries (ignoring the pivot rule) and utilizing $\mathcal{O}(k^2 N)$ additional operations.
    



\subsection{Pivot rules}
\label{sec:qr-pivot-rules}

The goal of column-pivoted partial QR is to identify the
``most important'' columns of a rectangular matrix.
There is an extensive literature on
pivot selection rules for QR decompositions (see~\cite{GE96:Efficient-Algorithms} and the references therein).
The most sophisticated methods, called strong rank-revealing QR algorithms, enjoy powerful approximation guarantees but are intricate and computationally expensive. 

We can also consider simpler strategies that are computationally cheaper, akin to the diagonal
pivot rules used in pivoted partial Cholesky.  If $\mat{A} = \mat{B}^* \mat{B}$,
then the diagonal entries of $\mat{A}$ agree with the
squared column norms of $\mat{B}$.  That is,
\[
\mat{A}(j,j) = \Vert \mat{B}(:, j) \Vert_2^2
\quad\text{for each $j = 1, \dots, N$.}
\]
This correspondence allows us to equip partial QR with
analogs of the partial Cholesky pivoting strategies.
In particular, greedy pivot rules are classical~\cite[\S 1]{high90c}.

From our current vantage, it is natural to
consider a \textit{randomly pivoted QR method}
(\Cref{alg:adaptive} in the appendix with $T = 1$) that is akin to randomly
pivoted Cholesky (\Cref{alg:rpcholesky}).
At each step $i$ of this method,
we sample the next pivot $s_{i}$ in proportion
to the squared column norms of the residual $\mat{B}^{(i-1)}$.

\subsection{Blocking} \label{sec:block}

A standard strategy for accelerating (column-pivoted) QR methods
is to select a block of columns to eliminate
at each step~\cite[Sec.~5.2.3]{GVL13:Matrix-Computations-4ed}.
In particular, we can consider a blocked variant of
randomly pivoted QR (\Cref{alg:adaptive} in the appendix with $T > 1$).
Let $T\ge 1$ be a block size parameter.
At each step $i$, we sample $T$ pivot columns independently
with probability proportional to the squared column norms
of the residual $\mat{B}^{(i-1)}$.  We project out these
columns and repeat.

Analogously, we can develop a blocked version of
\RPCholesky (\Cref{alg:rpcholesky_blocking} with $T > 1$).
At each step $i$, we sample $T$ pivot columns independently
with probability proportional to the diagonal entries
of the residual $\mat{A}^{(i-1)}$.  We eliminate these
columns and repeat.

Both of these methods require careful implementation
to manage potential issues with numerical stability.
In particular, block \RPCholesky repeatedly computes a full Cholesky decomposition of a square matrix that might be numerically rank-deficient.
We can address this issue by adding a positive multiple of the identity to this matrix; see \Cref{alg:rpcholesky_blocking} for details.
Block randomly pivoted QR, on the other hand, can suffer from loss of orthogonality in the computed $\mat{Q}$ matrix; see 
a standard numerical linear algebra reference (e.g., \cite[Ch.~5]{GVL13:Matrix-Computations-4ed}) for discussion on stably computing a QR decomposition.

How does block \RPCholesky compare to simple \RPCholesky?
Usually, the block \RPCholesky method is faster.
When block \RPCholesky with block size $T=50$ is applied to a $10^5 \times 10^5$ kernel matrix (see \Cref{sec:krr}), we obtain a $5\times$ speedup using the $\ell_1$ Laplace kernel \eqref{eq:l1_laplace} and a $20\times$ speedup after switching to the Gaussian kernel.
As demonstrated in \Cref{tab:many-matrices}, the approximation quality of block \RPCholesky and simple \RPCholesky is similar for many inputs.
However, block \RPCholesky can produce a significantly worse approximation on some inputs, leading to a $10^6\times$ higher trace-norm error when applied to the \texttt{COMET\_MC\_SAMPLE} example, for example.
The theoretical and empirical properties of block \RPCholesky are studied in the followup paper \cite{ETW24a}, which introduces an ``accelerated \RPCholesky'' method that remedies the deficits of block \RPCholesky.

\subsection{Randomly pivoted QR: Origins}

We believe that the randomly pivoted QR method was
first proposed in the theoretical computer
science literature on column subset selection problems.
In 2004, Frieze et al.~\cite{FKV04} studied projection approximations
where a set of column indices is chosen randomly by sampling in
proportion to the squared column norms of the input matrix.
In 2006, Deshpande et al.~\cite{DRVW06, DV06} described a procedure
that applies the Frieze et al.~approximation iteratively,
projecting out the contributions of previously selected columns
at each step.
They called the resulting method ``adaptive sampling'',
in contrast to the one-shot sampling method of Frieze et al.
Their approach is essentially the same as block randomly pivoted QR,
modulo implementation details.
As we will explain, the blocking is central to their proposal.
We have chosen to use the terminology
``(block) randomly pivoted QR''
in this work because it clarifies the relationship with standard
linear algebra algorithms.

\begin{algorithm}[t]
  \caption{\RPCholesky: Block variant} \label{alg:rpcholesky_blocking}
  \textbf{Input:} Psd matrix $\mat{A} \in \mathbb{C}^{N\times N}$; block size $T$; tolerance $\eta$ or approximation rank $k$ which is a multiple of $T$
  
  \textbf{Output:} Pivot set $\set{S}$; matrix $\mat{F}$ defining Nystr\"om approximation $\mat{\hat{A}} = \mat{F}\mat{F}^*$ 
  \begin{algorithmic}
    \State Initialize $\mat{F} \leftarrow \mat{0}_{N \times k}$,
    $\set{S} \leftarrow \emptyset$, and
    $\vec{d} \leftarrow \diag \mat{A}$
    \Comment{Evaluate diagonal of input matrix}
    \For{$i = 0$ to $k/T - 1$}
    \Comment{Alternatively, run until $\sum_{j=1}^N \vec{d}(j) \le \eta \tr \mat{A}$}
    \State Sample $s_{iT + 1}, \ldots, s_{iT + T} \stackrel{\rm iid}{\sim} \vec{d} / \sum_{j=1}^N \vec{d}(j)$
    \Comment{Probability prop. to diagonal elements of residual}
    \State $\set{S}' \leftarrow \Call{Unique}{\{s_{iT + 1}, \ldots, s_{iT + T}\}}$
    \State $\set{S} \leftarrow \set{S} \cup \set{S}'$
    \State $\mat{G} \leftarrow \mat{A}(:,\set{S}')$
    \Comment{Evaluate columns $\set{S}'$ of input matrix}
    \State $\mat{G} \leftarrow \mat{G} - \mat{F} \mat{F}(\set{S}', :)^*$
    \Comment{Remove overlap with previously chosen columns}
    \State $\mat{R} \leftarrow \Call{Chol}{\mat{G}(\set{S}',:) + \varepsilon_{\rm mach}\tr(\mat{G}(\set{S}',:))\mathbf{I}}$ \Comment{Stabilized Cholesky $\mat{G}(\set{S}',:) \approx \mat{R}^*\mat{R}$}
    \State $\mat{F}(:,iT + 1:iT + |\set{S}'|) \leftarrow \mat{G} \mat{R}^{-1}$
    \Comment{Update approximation}
    \State $\vec{d} \leftarrow \vec{d} - \Call{SquaredRowNorms}{\mat{G} \mat{R}^{-1}}$
    \Comment{Track diagonal of residual matrix}
    \State $\vec{d}\leftarrow \max\{\vec{d},\vec{0}\}$ \Comment{Ensure diagonal remains nonnegative}
    \EndFor
    \State  Remove zero columns from $\mat{F}$
  \end{algorithmic}
\end{algorithm}

\subsection{Randomly pivoted QR: Theory}
\label{sec:rpqr-theory}

In 2006, the original papers~\cite{DRVW06,DV06} on
randomly pivoted QR (i.e., adaptive sampling)
focused on proving approximation guarantees.
Here is a typical result.

\begin{proposition}[\protect{Deshpande et al.~\cite[Thm.~1.2]{DRVW06}}]
\label{prop:drvw}
Fix a matrix $\mat{B} \in \real^{M \times N}$,
a target approximation rank $r$,
and a tolerance $\varepsilon \in (0,1)$.
Set the block size $T \geq r/\varepsilon$.
After $s$ steps, the block randomly pivoted
QR method (\Cref{alg:adaptive}) produces
a random projection approximation $\mat{\hat{B}}$
with rank $k = s\cdot T$ that satisfies
\[
\mathbb{E} \Vert \mat{B} - \mat{\hat{B}} \Vert_{\mathrm{F}}^2
    \leq (1 - \varepsilon)^{-1} \Vert \mat{B} - \lowrank{\mat{B}}_r \Vert_{\mathrm{F}}^2 + \varepsilon^s \Vert \mat{B} \Vert_{\mathrm{F}}^2.
\]
\end{proposition}

Using \Cref{prop:chol-qr}, we obtain a parallel result
for block \RPCholesky (\Cref{alg:rpcholesky_blocking}).  Fix a psd input matrix $\mat{A} \in \real^{N \times N}$,
an approximation rank $r$, and a tolerance $\varepsilon \in (0,1)$.
Set the block size $T \geq r/\varepsilon$.
After choosing $k/T$ blocks of columns, block \RPCholesky
produces a random Nystr{\"o}m approximation $\mat{\hat{A}}$
with rank $k$ that satisfies
\[
\mathbb{E} \tr( \mat{A} - \mat{\hat{A}} )
    \leq (1 - \varepsilon)^{-1} \tr( \mat{A} - \lowrank{\mat{A}}_r )
    + \varepsilon^{k/T} \tr(\mat{A}).
\]
Assuming $\varepsilon \leq 1/2$, this result guarantees an $(r, 3\varepsilon)$-approximation
with rank
\[
k = \frac{r}{\varepsilon} + r \cdot \frac{\log(1/\eta)}{\varepsilon \log(1/\varepsilon)}
\quad\text{where}\quad
\eta = \frac{\tr(\mat{A} - \lowrank{\mat{A}}_r)}{\tr(\mat{A})}.
\]
For comparison, our result \Cref{thm:simple_bound} guarantees an $(r,\varepsilon)$-approximation with approximation rank
$k = r/\varepsilon + r \log(1/(\eta \varepsilon))$,
which is always better.

More seriously, Proposition~\ref{prop:drvw} requires
the user to fix the approximation rank $r$ and error
tolerance $\varepsilon$ in advance.
The block size $T$ is adapted to both parameters.
Furthermore, the statement is vacuous
in case the block size $T = 1$.
In other words, the existing theory is silent about the
versions of these algorithms that we recommend for use in practice, which have fixed block size and stopping rules that depend
on the observed error.

Our paper provides the first analysis of \RPCholesky
that does not require unrealistic parameter choices
and that addresses the fundamental case where the block size $T = 1$.
See \Cref{sec:theoretical-analysis} for the details, including a new error bound (\Cref{cor:main_bound}) for randomly pivoted QR with $T = 1$.
In the followup paper \cite{ETW24a}, we extend our analysis
to handle the case of a fixed block size $T > 1$.

\subsection{Randomly pivoted QR: Empirical work}

In the period from 2009--2013,
researchers in applied machine learning explored the
empirical performance of randomly pivoted QR for
approximating kernel matrices.
In sharp contrast to our findings for \RPCholesky (\Cref{sec:rpc}),
their conclusions were pessimistic.

We must stress that the existing numerical
work applies randomly pivoted QR to a kernel matrix to obtain a column projection
approximation at a cost of $\mathcal{O}(k N^2)$ operations
(they run \cref{alg:adaptive} with input $\mat{A}$).
\textbf{These studies do not apply \RPCholesky
to obtain a column Nystr{\"o}m
approximation at the lesser cost of $\mathcal{O}(k^2 N)$ operations.}

In their review of kernel approximation, Kumar et al.\
reported~\cite[Tab.~3]{KMT12}
that randomly pivoted QR (\Cref{alg:adaptive})
was $60\times$ to $800\times$ slower than the Nystr{\"o}m
approximation~\eqref{eq:nystrom} with uniformly sampled
columns.  For data sets with
$N > 4000$ data points, they did not even run the
randomly pivoted QR algorithm because they considered
it impractical.  Summarizing their findings, they
emphasized ``the computational and storage burdens''
(p.~990) and they 
stated
that the algorithm
``requires a full pass through [the kernel matrix] $\mat{K}$
at each iteration and is thus inefficient for large $\mat{K}$''
(p.~989).

The literature describes some attempts~\cite{KMT09b,KMT09a,KMT12,WZ13}
to improve randomly pivoted QR, but the
algorithm fell into disuse over the subsequent decade.

\subsection{Randomly pivoted Cholesky: Origins}

In 2017, Musco \& Woodruff~\cite[p.~3]{MW17} briefly noted that
one can perform ``adaptive sampling'' more efficiently,
given access to the Gram matrix.
Their observation suggests an algorithm similar to
block \RPCholesky (\Cref{alg:rpcholesky_blocking}).
This paper does not include an implementation or report any numerical experiments.

Randomly pivoted Cholesky also appears in a 2020 paper of Poulson \cite{Pou20}.
Rather than using \RPCholesky for low-rank approximation, Poulson uses \RPCholesky to sample from a projection DPP.
In Poulson's work, the input matrix $\mat{A}$ is always a rank-$k$ orthoprojector, \RPCholesky is always run for exactly $k$ steps, and the computational output is the set $\set{S}$ of pivots; the factor $\mat{F}$ is discarded.

To the best of our knowledge, the papers~\cite{MW17,Pou20}
are the sole references to \RPCholesky in the literature.
Neither paper documents numerical experiments or provides a theoretical analysis of \RPCholesky for the low-rank approximation task.

\subsection{Comparison with random Fourier features}
Random Fourier features (RFF) \cite{RR07} is a popular Monte Carlo method for the low-rank approximation of a psd kernel matrix (\Cref{sec:kernel_basics}) associated with a translationally invariant kernel function.
The main disadvantage of RFF that the approximation error $\mat{A} \approx \mat{\hat{A}}$ converges at the Monte Carlo rate $\order(k^{-1/2})$, whereas \RPCholesky is \emph{spectrally accurate}, producing approximations comparable to a best low-rank approximation.
For applications where a high-accuracy approximation is important, RFF fares significantly worse than Nystr\"om methods like \RPCholesky \cite{diaz2023robust}.
However, RFF can be useful in applications where a fast, crude approximation of $\mat{A}$ is sufficient.

\section{Applications to Kernel Machine Learning}
\label{sec:rpc}

In this section, we undertake
a numerical study to evaluate the performance of \RPCholesky
on benchmark kernel computations from scientific machine 
learning.
\Cref{sec:krr} treats a kernel ridge regression problem that arises in quantum chemistry, and \Cref{sec:spectral_clustering} discusses a kernel spectral clustering problem from molecular biophysics.
%

\subsection{Kernel methods: Basics} \label{sec:kernel_basics}

Kernel methods~\cite{SS02} are designed for analyzing data in a general 
domain $\mathcal{X}$, equipped with a \textit{kernel function}
$K : \mathcal{X} \times \mathcal{X} \to \complex$.
We interpret the kernel function as a measure of similarity
between a pair of data points.
Suppose that $\vec{x}_1, \dots, \vec{x}_N \in \mathcal{X}$
is a list of $N$ data points.  We can form a
\textit{kernel matrix} $\mat{A} \in \complex^{N \times N}$
that tabulates the pairwise similarities:
\[
\mat{A}(i,j) \coloneqq K(\vec{x}_i, \vec{x}_j)
\quad\text{for $1 \leq i, j \leq N$.}
\]
We say that the kernel function $K$ is \textit{positive definite}
if the kernel matrix $\mat{A}$ is psd for every family of
$N$ data points in $\mathcal{X}$ and every natural number $N$.
For example, the inner-product kernel and the Gaussian kernel
are both positive-definite kernels on $\complex^d$:
\[
\begin{aligned}
K(\vec{x}, \vec{y}) &= \vec{x}^* \vec{y}
&&\quad\text{(inner-product kernel)}; \\
K(\vec{x}, \vec{y}) &= \exp\biggl(-\tfrac{1}{2 \sigma^2} \lVert \vec{x} - \vec{y} \rVert_2^2 \biggr)
&&\quad\text{(Gaussian kernel)}.
\end{aligned}
\]
The parameter $\sigma > 0$ is called the \textit{bandwidth}
of the Gaussian kernel.
Kernel methods reduce data analysis tasks in $\mathcal{X}$
to linear algebra computations on the kernel matrix $\mat{A}$.

We have seen that \RPCholesky 
quickly and reliably
finds a 
low-rank approximation of
a psd kernel matrix.
Equivalently, the pivots $\set{S}$ of \RPCholesky identify a modest number of 
data points $\{\vec{x}_s : s \in \set{S} \}$ that can be used to summarize the
data set.
Therefore, we can incorporate \RPCholesky into the computational pipeline to obtain more scalable algorithms for kernel machine learning.
We will elaborate on this idea in the next two sections.

\subsection{Kernel ridge regression} \label{sec:krr}

One powerful application for \RPCholesky is to accelerate \emph{kernel ridge regression} (KRR)~\cite[\S4.9.1]{SS02}.  We will study an application
in quantum chemistry.

\subsubsection{Functional regression}

KRR is a nonlinear extension of least-squares regression that approximates an unknown input--output map
using a positive-definite kernel function $K : \mathcal{X} \times \mathcal{X} \to \complex$
and input--output pairs $(\vec{x}_1, y_1),\ldots,(\vec{x}_N, y_N)\in \mathcal{X}\times \complex$.
As output, KRR provides a prediction function of the form
\begin{equation}
    \label{eq:krr_ansatz}
    f(\cdot \, ; \, \vec{\beta}) \coloneqq \sum_{i=1}^N \beta_i K\bigl(\vec{x}_i, \cdot\bigr),
\end{equation}
with the coefficient vector $\vec{\beta}$ chosen to minimize a regularized least-squares loss:
\begin{equation*} \label{eq:minimization}
    \min_{\vec{\beta} \in \mathbb{C}^N} \frac{1}{N}
    \sum_{j=1}^N \biggl\lvert 
    f\bigl(\vec{x}_j\, ;\,\vec{\beta}\bigr) - y_j
    \biggr\rvert^2 
    + \lambda\, \sum_{i,j = 1}^N \beta_i \beta_j K\bigl(\vec{x}_i, \vec{x}_j\bigr).
\end{equation*}
Explicitly, 
the vector $\vec{\beta}$ is the solution to a linear system
\begin{equation} \label{eq:krr_linear_system}
    \vec{\beta} = (\mat{A} + \lambda N \, \Id)^{-1} \vec{y},
\end{equation}
where $\mat{A}$ is the $N\times N$ kernel matrix
induced by the input data $\vec{x}^{(i)}$
and $\vec{y}$ is the vector of output values.

\subsubsection{Restricted KRR via \RPCholesky}

Directly computing the vector $\vec{\beta}$ 
via \eqref{eq:krr_linear_system} would require solving
a dense linear system at $\order(N^3)$ cost. 
As a faster alternative, we can solve a \emph{restricted} version of the KRR problem at $\order(k^2N)$ cost, where $k$ is a user-definable parameter.
Restricted KRR was proposed by Smola and Bartlett \cite{smola2000sparse} and developed in subsequent works \cite{rudi2017falkon,RCCR18,diaz2023robust}.

In restricted KRR, we first identify a set $\set{S} = \{s_1, \ldots, s_k\}\subseteq \{1,\ldots,N\}$ of $k$ \emph{landmarks} that achieve good coverage over the data set.
Next, we find a restricted prediction function
\begin{equation} \label{eq:krr_acc}
    \hat{f}\bigl(\cdot \, ; \, \vec{\hat{\beta}}\, \bigr) 
    = \sum_{i = 1}^k \hat{\beta}_i K\bigl(\vec{x}_{s_i}, \cdot\bigr)
\end{equation}
by solving the regularized least-squares problem
\begin{equation*}
    \min_{\vec{\hat{\beta}}\in\complex^k} \frac{1}{N}
    \sum_{j=1}^N \biggl\lvert  \hat{f}\bigl(\vec{x}_j \, ; \, \vec{\hat{\beta}}\, \bigr) - y_j
    \biggr\rvert^2 
    + \lambda\, \sum_{i,j = 1}^k \hat{\beta}_i \hat{\beta}_j K\bigl(\vec{x}_{s_i}, \vec{x}_{s_j}\bigr).
\end{equation*}
The coefficient vector $\vec{\hat{\beta}} \in \mathbb{C}^k$ is given by a smaller linear system
\begin{equation*}
    \vec{\hat{\beta}} = \bigl( \mat{A}(\set{S},:)\mat{A}(:,\set{S})
    + \lambda N\, \mat{A}(\set{S},\set{S}) \bigr)^{-1} 
    \mat{A}(\set{S}, :) \vec{y}.
\end{equation*}
involving a $k \times k$ matrix, which is relatively inexpensive to solve.
Forming and solving this system requires $\order(k^2N)$ operations.
Evaluating the prediction function \eqref{eq:krr_acc} for restricted KRR requires just $\order(k)$ operations,
which improves on the $\order(N)$ cost of evaluating \eqref{eq:krr_ansatz}.

In past work, the landmark set $\set{S}$ has been selected by uniform random sampling \cite{rudi2017falkon,diaz2023robust}, ridge leverage score sampling \cite{RCCR18}, or greedy procedures \cite{smola2000sparse}.
To improve on these approaches, we propose choosing the landmarks $\set{S}$ to be the pivot set chosen by \RPCholesky, resulting in the \RPCholesky-accelerated KRR method shown in \Cref{alg:krr}.
We demonstrate below that the \RPCholesky-based approach leads to improved out-of-sample prediction accuracy compared to previous landmark selection approaches.

\begin{algorithm}[t]
  \caption{\RPCholesky-accelerated kernel ridge regression} \label{alg:krr}
  \textbf{Input:} Data points $\set{X} = \{\vec{x}_1,\ldots,\vec{x}_N\} \subseteq \mathbb{C}^d$; output values $\vec{y}\in\mathbb{C}^N$; approximation rank $k$; regularization parameter $\lambda > 0$
  
  \textbf{Output:} Pivots $\set{S}$ and coefficients $\vec{\hat{\beta}}$ defining a prediction function $\hat{f}(\cdot)$ by \eqref{eq:krr_acc}
  \begin{algorithmic}
    \State $\mat{A} \leftarrow \Call{KernelMatrix}{\set{X}}$
    \State $(\sim,\set{S}) \leftarrow \Call{RPCholesky}{\mat{A},k}$
    \State $\vec{\hat{\beta}} \leftarrow \bigl(\mat{A}(\set{S},:)\mat{A}(:,\set{S}) + \lambda N \, \mat{A}(\set{S},\set{S})\bigr)^{-1} \mat{A}(\set{S},:) \vec{y}$
  \end{algorithmic}
\end{algorithm}

\subsubsection{QM9 data}

To showcase the effectiveness of \RPCholesky-accelerated KRR, we use \Cref{alg:krr} to predict the highest occupied molecular orbital (HOMO) energy of organic molecules from the QM9 data set \cite{RDRv14,RvBR12}.
The HOMO energy quantifies the electron-donating capacity of a molecule, and it is traditionally obtained using expensive first-principles calculations.
As a modern alternative, a recent ``Editor's Pick'' journal article \cite{STR+19} proposes applying KRR to predict HOMO energies, and the authors train their prediction function on the QM9 data set due to its large size and molecular diversity.
Here, we evaluate the \RPCholesky-accelerated KRR method on the HOMO prediction task with the QM9 data.

To represent molecules as vectors for KRR, we use a standard feature set based on the Coulomb repulsions between the atomic nuclei and the nuclear charges \cite[{\S}III.A]{STR+19}.
We standardize the data and, following  \cite{STR+19}, evaluate the similarity between data points using the positive-definite $\ell_1$ Laplace kernel:
\begin{equation} \label{eq:l1_laplace}
    K(\vec{x},\vec{y}) = \exp \biggl(  - \tfrac{1}{\sigma} \sum\nolimits_{j=1}^d \bigl\lvert \vec{x}(j) - \vec{y}(j) \bigr\rvert \biggr).
\end{equation}
We divide the data into
100,000 data points for training and roughly 33,000 data points for testing and set the bandwidth $\sigma$ and the ridge parameter $\lambda$ using cross-validation.

Forming and storing the full kernel matrix for the QM9 data set would require 40 GB and 4 trillion arithmetic operations (20,000$\times$ the operation count for the \textbf{Smile} matrix from \Cref{sec:compare}).
Because of this high computational cost,
previous authors \cite{STR+19} applied KRR using a random subsample of $N \leq \text{64,000}$ training points,
but they remarked that the approximation quality improves with the size of the data (see their Fig.~8).
In contrast, our \RPCholesky-accelerated computational approach allows us to use all $N = \text{100,000}$ training points at a modest computational cost (5 mins on a laptop computer for $k = 1000$).

\begin{figure}[t]
    \begin{subfigure}{0.49\textwidth}
    \centering
    \includegraphics[width = \textwidth]{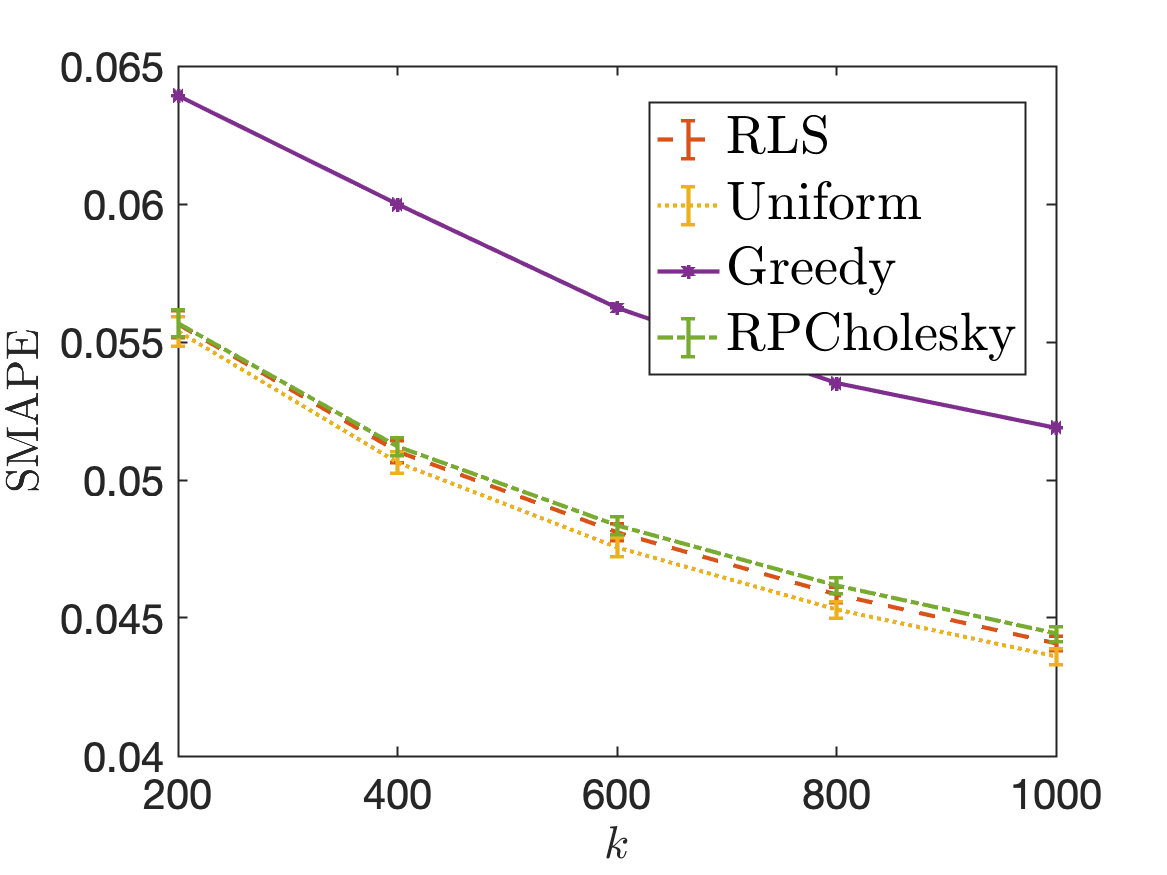}
    \end{subfigure}
    \hfill
    \begin{subfigure}{0.49\textwidth}
    \centering
    \includegraphics[width = \textwidth]{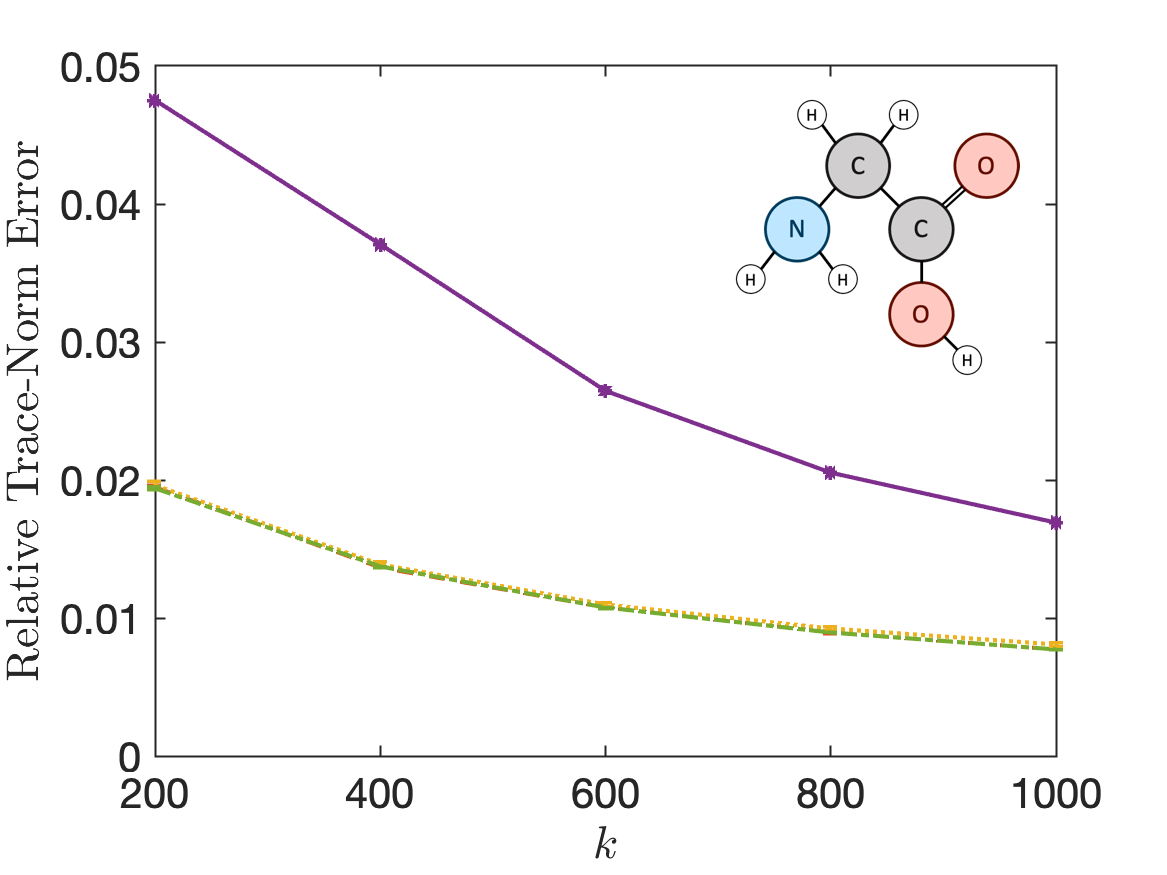}
    \end{subfigure}
    \caption{\textbf{Kernel ridge regression for QM9 data.}
    \textit{Left:} Prediction error \eqref{eq:smape} for several Nystr\"om algorithms.
    \textit{Right:} Relative trace-norm error.}
    \label{fig:chemistry}
\end{figure}

\Cref{fig:chemistry} displays the results for the test portion of the data set ($N_{\rm test} = 3.3 \times 10^4$ data points).
The out-of-sample prediction errors are measured using
the symmetric mean absolute percentage error:
\begin{equation} \label{eq:smape}
    \operatorname{SMAPE} = \frac{1}{N_{\rm test}} \sum_{i=1}^{N_{\rm test}} \frac{\biggl| y^{\rm test}_i - \hat{f}\bigl(\vec{x}_i^{\rm test} \, ; \, \vec{\hat{\beta}} \bigr)  \biggr|}{\tfrac{1}{2} \biggl| y^{\rm test}_i\biggr| + \tfrac{1}{2} \biggr| \hat{f}\bigl(\vec{x}_i^{\rm test} \, ; \, \vec{\hat{\beta}} \bigr) \biggr|}.
\end{equation}
The SMAPE for \RPCholesky, uniform sampling, and RLS sampling are similar, with uniform sampling being the slight favorite.
The greedy method has notably worse performance than the three other methods, and we were unable to use DPP sampling because of the large values $N = 10^5$ and $k = 10^3$.

\begin{table}[t]
    \centering
    \caption{\textbf{Out-of sample prediction for QM9 data.} Symmetric absolute percentage error for the nine largest molecules in the test portion of the QM9 data set.
    The smallest errors in each row are marked in bold.
}
    \label{tab:big_molecules}
    \begin{tabular}{rrcccc} \toprule
        Compound \# & Composition & Uniform & Greedy & RLS & \RPCholesky\\ \midrule
        1,996 & \texttt{CC(C)CC1CO1} & 0.481 & \textbf{0.034} & 0.036 & 0.035 \\
        8,664 & \texttt{CC(CO)(C=O)C=O} & 0.475 & 0.039 & \textbf{0.030} & 0.033 \\
        13,812 & \texttt{CC(O)C1CC1CO} & 0.473 & 0.045 & 0.023 & \textbf{0.022} \\
        64,333 & \texttt{CC1(C)CCOC(=N)O1} & 0.538 & \textbf{0.032} & 0.084 & 0.077 \\
        81,711 & \texttt{OC1C2CC1OCCO2} & 0.536 & \textbf{0.007} & 0.019 & 0.018 \\
        109,816 & \texttt{CCC1C(O)C1CC\#C} & 0.516 & 0.062 & 0.018 & \textbf{0.016} \\
        118,229 & \texttt{COCC(C)OC(C)C} & 0.529 & 0.051 & 0.018 & \textbf{0.017} \\
        122,340 & \texttt{CC1C(CCCO)N1C} & 0.533 & \textbf{0.002} & 0.035 & 0.027 \\
        131,819 & \texttt{OCCCN1C=NC=N1} & 0.486 & 0.081 & 0.028 & \textbf{0.026} \\ \midrule 
        Average & & 0.507 & 0.039 & 0.033 & \textbf{0.030}
    \end{tabular}
\end{table}

Nevertheless, \RPCholesky leads to higher accuracy than uniform sampling for the \emph{extremal} molecules in the test set which have the greatest number of atoms (29 atoms).
Across these nine molecules, 
\Cref{tab:big_molecules} shows that
\RPCholesky is 10\% to 30\% more accurate than RLS and greedy pivoting, and \RPCholesky achieves \textbf{17$\times$ smaller prediction errors} than uniform sampling.
This observation suggests that \RPCholesky is more effective at representing less populated regions of data space, as seen earlier in the \textbf{Smile} example (\Cref{sec:compare}).
The importance of sampling diverse data points in kernel ridge regression to boost outlier predictive performance
was also emphasized in the recent work \cite{FSS21}.

\subsection{Kernel spectral clustering} \label{sec:spectral_clustering}

We can also use \RPCholesky to accelerate \emph{kernel spectral clustering} \cite{SM00,von07a}.  We will study an application in molecular biophysics.

\subsubsection{Kernel clustering}

In kernel spectral clustering, we use a positive-definite kernel function $K : \mathcal{X} \times \mathcal{X} \to \complex$ to compute similarities between data points $\vec{x}_1,\ldots,\vec{x}_N$.
Then we find a low-dimensional embedding $\bm{V} \in \mathbb{C}^{N \times m}$ of the $N$ data points into $m$-dimensional Euclidean space that preserves the kernel-based similarities as well as possible.
Afterward, we apply the conventional $k$-means algorithm \cite{AV07} to cluster the rows of $\bm{V}$.

Specifically, the embedding matrix $\bm{V}$ is chosen to minimize the kernel-based distortion
\begin{equation*}
    \frac{1}{2} \sum\nolimits_{i,j = 1}^N K(\bm{x}_i, \bm{x}_j) \lVert \bm{V}(i, :) - \bm{V}(j, :) \rVert^2
\end{equation*}
while also satisfying the isotropy condition:
\begin{equation*}
    \sum\nolimits_{i = 1}^N \left( \sum\nolimits_{j=1}^N K(\bm{x}_i, \bm{x}_j)\right) \bm{V}(i, :)^* \bm{V}(i, :) = \mathbf{I}.
\end{equation*}
The exact solution is described in~\cite{belkin2003laplacian}.
We construct the symmetrically scaled transition matrix $\mat{H} = \mat{D}^{-1/2} \mat{A} \mat{D}^{-1/2}$, where $\mat{A} \in \mathbb{C}^{N \times N}$ is the kernel matrix
and $\mat{D} \in \mathbb{C}^{N \times N}$ is the diagonal matrix that lists the row sums of $\mat{A}$.  Then we calculate the $m$ dominant eigenvectors $\mat{U} = \begin{bmatrix} {\vec{u}}_1 & \cdots & {\vec{u}}_m \end{bmatrix} \in \complex^{N \times m}$ of the transition matrix $\mat{H}$.
The optimal embedding matrix $\bm{V}$ is obtained from the diagonal rescaling $\mat{V} = \bm{D}^{-1/2} \mat{U}$.
We note that there are several nonequivalent versions of spectral clustering; we use the present version because it is effective and widely used in biochemistry \cite{GHR21,RZMC11} and amenable to acceleration by \RPCholesky.

\subsubsection{Accelerated kernel clustering via \RPCholesky}

Directly computing the eigendecomposition of $\mat{H} = \mat{D}^{-1/2} \mat{A} \mat{D}^{-1/2}$ would require $\order(N^3)$ operations.
However, there is a faster approach due to Fowlkes et al.~\cite{fowlkes2004spectral} that requires just $\order(k^2 N)$ operations, where $k$ is a parameter.

In this approach, we replace the kernel matrix $\mat{A}$ with a rank-$k$ approximation $\mat{\hat{A}}^{(k)}$ and replace
the diagonal matrix $\mat{D}$ with the diagonal matrix $\mat{\hat{D}}$ listing the row sums of $\mat{\hat{A}}^{(k)}$.
We form the approximate transition matrix $\mat{\hat{H}} = \mat{\hat{D}}^{-1/2} \mat{\hat{A}}^{(k)} \mat{\hat{D}}^{-1/2}$
and compute its $m$ dominant eigenvectors
$\mat{\hat{U}} = \begin{bmatrix} \hat{\vec{u}}_1 & \dots & \hat{\vec{u}}_m \end{bmatrix}$.
Just as before,
we obtain an embedding $\mat{\hat{V}} = \mat{\hat{D}}^{-1/2} \mat{\hat{U}}$,
and we apply k-means clustering to the rows of $\mat{\hat{V}}$.

Fowlkes et al.~\cite{fowlkes2004spectral} used Nystr\"om approximation with uniform sampling to obtain the low-rank approximation $\mat{\hat{A}}^{(k)}$.  We propose to replace uniform sampling with \RPCholesky, and we will demonstrate that this modification can significantly enhance the clustering accuracy.  Our \RPCholesky-accelerated spectral clustering algorithm is presented in \Cref{alg:spectral_clustering}.

\begin{algorithm}[t]
  \caption{\RPCholesky-accelerated spectral clustering} \label{alg:spectral_clustering}
  \textbf{Input:} Data points $\set{X} = \{\vec{x}_1,\ldots,\vec{x}_N\} \subseteq \complex^d$; eigenvector count $m$; approximation rank $k$
  
  \textbf{Output:} Partition of labels $\{1, \ldots, N\}$ into clusters $\set{C}_1\cup \cdots \cup \set{C}_c$
  \begin{algorithmic}
    \State $\mat{A} \leftarrow \Call{KernelMatrix}{\set{X}}$
    \State $\mat{F} \leftarrow \Call{RPCholesky}{\mat{A},k}$
    \State $\mat{\hat{D}}\leftarrow \diag(\mat{F}(\mat{F}^* \onevec))$
    \State $\mat{G} \leftarrow \mat{\hat{D}}^{-1/2}\mat{F}$
    \State $(\mat{U},\sim,\sim) \leftarrow \Call{SVD}{\mat{G}, \texttt{`econ'}}$
    \State $\mat{\hat{V}} \leftarrow \mat{\hat{D}}^{-1 \slash 2}\mat{U}$
    \State $\set{C}_1,\ldots,\set{C}_c \leftarrow$ \Call{k-means}{$\mat{\hat{V}}(:,1:m)$}
    \Comment{Cluster the rows of $\mat{\hat{V}}(:,1:m)$}
  \end{algorithmic}
\end{algorithm}

\subsubsection{Alanine dipeptide trajectories}

Kernel spectral clustering has become a popular approach for interpreting molecular dynamics (MD) data sets in computational biochemistry \cite{GHR21}.
A typical MD data set consists of the $(x, y, z)$-positions of the backbone (non-hydrogen) atoms in a simulated biomolecule.
Spectral clustering is used to identify the metastable (long-lived) conformations of the molecule, which help determine the molecular functionality.

Alanine dipeptide (\ce{CH_3-CO-NH-C_{\alpha}HCH_3-CO-NH-CH_3}) is a commonly studied molecule which has emerged as a benchmark when developing and testing numerical methods.
The metastable states of alanine dipeptide are well-described by the dihedral angles $\phi$ between 
\ce{C}, \ce{N}, \ce{C_{\alpha}}, and \ce{C} and $\psi$ 
between \ce{N}, \ce{C_{\alpha}}, \ce{C}, and \ce{N}.
Yet even without access to this $\phi$--$\psi$ feature space,
we can identify the metastable states by applying spectral clustering using the $(x, y, z)$-positions of the backbone atoms.

We downloaded one of the $250$ns alanine dipeptide trajectories documented in \cite{WN18}, which includes the positions of the backbone atoms at intervals of $1$ps, leading to $N = 2.5 \times 10^5$ data points in $\real^{30}$.
Because of the large size of the data set, it would be extremely expensive to apply spectral clustering directly.
To make these computations tractable, the biochemistry literature often applies subsampling to the data and then runs spectral clustering codes for a day or more on high-performance workstations \cite[Sec.~4]{RZMC11}.
Here we document an alternative approach using \RPCholesky-accelerated spectral clustering, which uses all $N = 2.5 \times 10^5$ data points while retaining a modest computational cost (20 seconds on a laptop computer for $k = 150$).

In our approach, we first quantify the similarity between alanine dipeptide configurations using a Gaussian kernel with bandwidth $\sigma = 0.1\,\text{nm}$.
We then apply \Cref{alg:spectral_clustering} to form a low-dimensional embedding and find clusters in the data.
Because the first three eigenvalues of $\mat{\hat{D}}^{-1/2} \mat{\hat{A}} \mat{\hat{D}}^{-1/2}$ are much larger than the rest, we cluster based on the dominant $m=3$ eigenvectors.
We use the k-means algorithm to identify $c=4$ clusters and present the results in \Cref{fig:clustering}.

\begin{figure}[t]
    \centering
    \begin{minipage}{.66\textwidth}
        \centering
        \includegraphics[scale = .45]{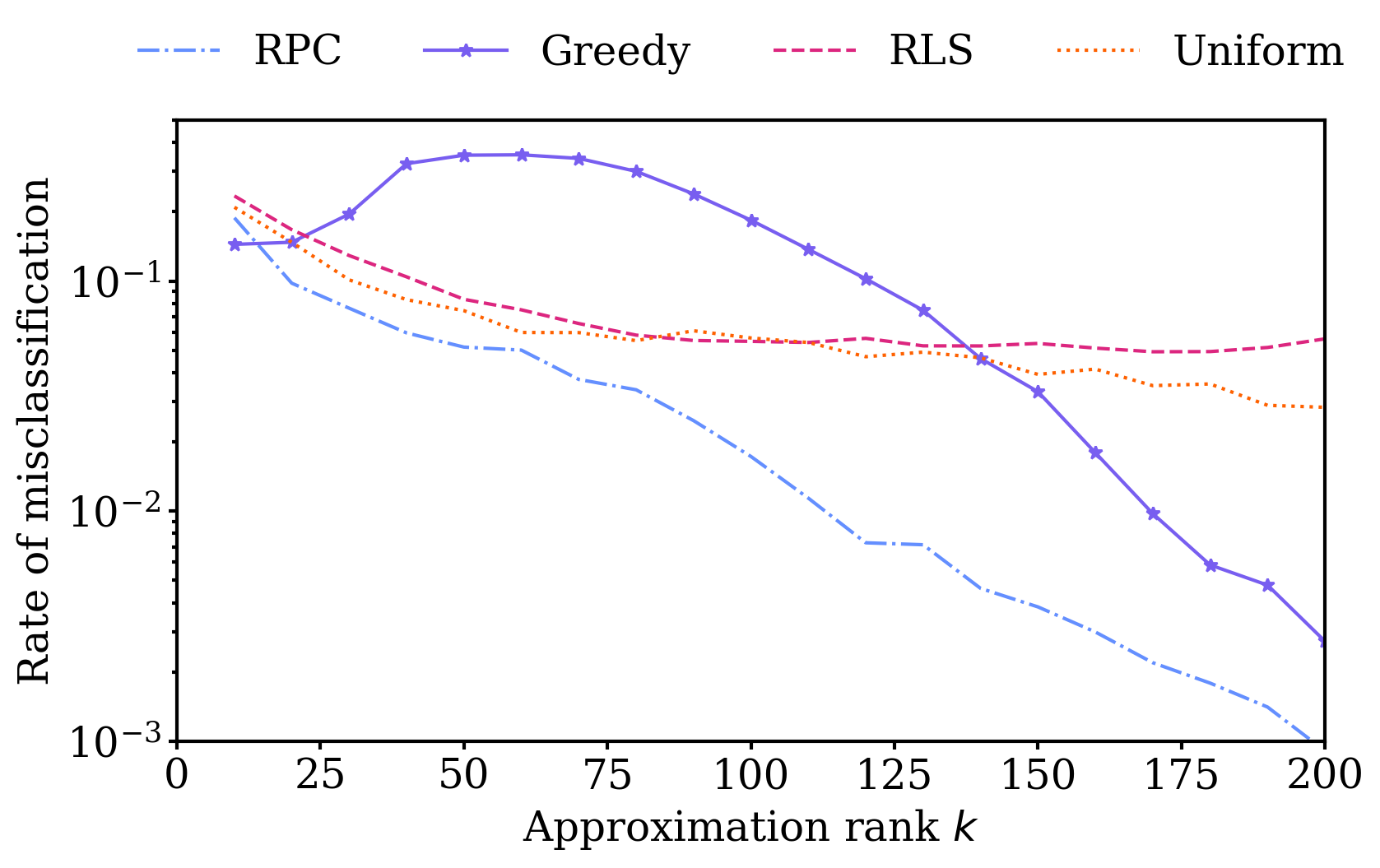}
    \end{minipage}
    \begin{minipage}{.32\textwidth}
        \centering
        \includegraphics[scale = .45]{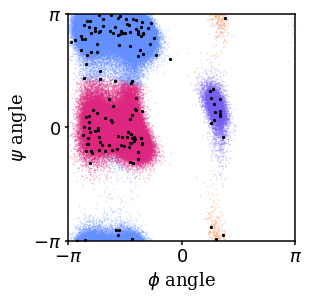}
        
        \RPCholesky
    \end{minipage}
    \vspace{.25cm}
    \newline
    \begin{minipage}{.32\textwidth}
        \centering
        \includegraphics[scale=.45]{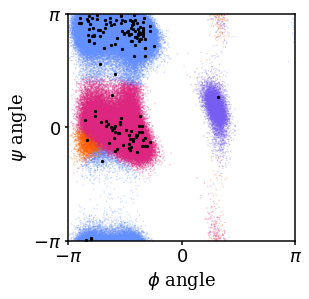}
        
        Uniform
    \end{minipage}
    \begin{minipage}{.32\textwidth}
        \centering
        \includegraphics[scale = .45]{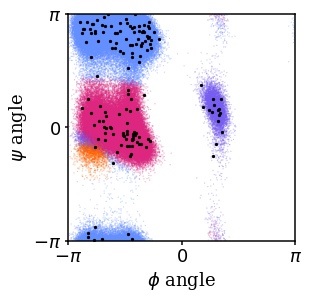}
        
        RLS
    \end{minipage}
    \begin{minipage}{.32\textwidth}
        \centering
        \includegraphics[scale = .45]{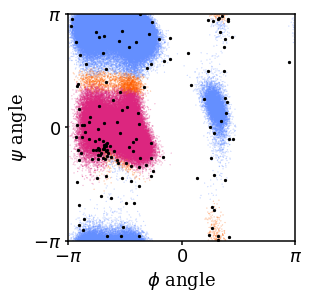}
        
        Greedy
    \end{minipage}
    \caption{\textbf{Spectral clustering for alanine dipeptide trajectories.}
    \textit{Top left:} Misclassification rate, averaged over 1000 independent trials.
    \textit{Top right:} Example of correct clustering ($<0.2\%$ misclassification) produced by \RPCholesky with rank $k = 150$.
    \textit{Bottom:} Incorrect clusterings ($>2\%$ misclassification) produced by uniform, RLS and greedy sampling with rank $k = 150$.
    Black dots mark data points selected as pivots.}
    \label{fig:clustering}
\end{figure}

To measure the error, we obtain reference clusters by running \RPCholesky with $k = 1000$ columns.
These reference clusters are consistent with~\cite[Fig.~4]{WN18}.
For each method and each approximation rank $10 \leq k \leq 200$, we calculate the fraction of incorrect labelings for the best permutation of the cluster labels.
We average over $1000$ independent trials and plot the resulting errors in \Cref{fig:clustering} (top left).
We find that \RPCholesky reliably produces a near-perfect clustering (top right panel) after reading just $k = 150$ of the $N = 2.5 \times 10^5$ columns.
In contrast, given the same approximation rank, uniform, RLS, and greedy sampling frequently produce an incorrect clustering (bottom panels).
With $k = 150$ columns, 
\textbf{\RPCholesky gives a 9$\times$ to 14$\times$ smaller misclassification rate} than uniform, RLS, or greedy sampling.

\subsubsection{Comparison with neural networks}

The recent journal article \cite{AJSW24a} compared an \RPCholesky-accelerated kernel method with a conventional neural network for the semisupervised clustering of alanine dipeptide trajectories.
Given a fixed number of parameters, the \RPCholesky-accelerated approach was found to improve the accuracy, speed, and interpretability.

\section{Theoretical Analysis of \RPCholesky}
\label{sec:theoretical-analysis}

Given the appealing computational profile and empirical performance of the \RPCholesky algorithm, we would like to understand when it is guaranteed to produce an accurate low-rank approximation. 
In this section, we will prove the following new result.

\begin{theorem}[Randomly pivoted Cholesky] \label{thm:main_bound}
Let $\mat{A}$ be a psd matrix.
Fix $r \in \mathbb{N}$ and  $\varepsilon > 0$.
The rank-$k$ column Nystr\"om approximation $\mat{\hat{A}}^{(k)}$ produced by
$k$ steps of \RPCholesky (\Cref{alg:rpcholesky})
attains the bound
\begin{equation*}
    \expect \tr \bigl(\mat{A} - \mat{\hat{A}}^{(k)} \bigr)
    \leq (1 + \varepsilon) \cdot \tr (\mat{A} - \lowrank{\mat{A}}_r),
\end{equation*}
provided that the number $k$ of columns satisfies 
\begin{equation} \label{eq:thm-k-bd}
    k \geq \frac{r}{\varepsilon} + \min\Biggl\{ r \log\biggl(\frac{1}{\varepsilon \eta}\biggr), r + r \log_+ \biggl(\frac{2^r}{\varepsilon}\biggr)\Biggr\}.
\end{equation}
The relative error $\eta$ is defined by $\eta \coloneqq \tr(\mat{A} - \lowrank{\mat{A}}_r) / \tr(\mat{A})$.
As usual, $\log_+(x) \coloneqq \max\{  \log x, 0 \}$ for $x > 0$,
and the logarithm has base $\mathrm{e}$.
\end{theorem}

\noindent
Let us emphasize that we do not need any prior knowledge to run \RPCholesky
and attain this approximation guarantee.  In fact, for any number $k$ of steps,
the error bound is valid for any pair ($r,\varepsilon$) that satisfies
the relation~\eqref{eq:thm-k-bd}.

The major takeaway from \Cref{thm:main_bound} is that \RPCholesky must produce
an $(r, \varepsilon)$-approximation as soon as the number $k$ of columns satisfies 
\begin{equation*}
\label{eq:main_bound}
    k \geq \frac{r}{\varepsilon} + r \log\biggl( \frac{1}{\varepsilon \eta} \biggr).
\end{equation*}
The logarithmic factor is typically a modest constant.
For example, $\log(1/\eta)$ is between $2.4$ and $26$ for all the examples in \Cref{tab:many-matrices}.
Therefore, this bound
is comparable with the minimal cost of $k \geq r / \varepsilon$ columns (\Cref{thm:lower-bound}).
Turning back to Table~\ref{tab:comparison}, we see that
\RPCholesky improves on the uniform sampling method,
where the number $k$ of columns can depend \textit{linearly}
on the relative error $1 / \eta$ (\Cref{sec:greedy}).
It also improves on the greedy method, in which the number $k$
of columns may be proportional to the dimension $N$ in the worst case (\Cref{sec:uniform}).  

\Cref{thm:main_bound} includes a second bound that demonstrates that \RPCholesky produces an $(r, \varepsilon)$-approximation
when the number $k$ of columns satisfies
\begin{equation}
\label{eq:alternative_bound}
    k \geq \frac{r}{\varepsilon} + r + r \log_+\biggl( \frac{2^r}{\varepsilon}\biggr).
\end{equation}
This alternative error bound is significant because it completely eliminates the dependence on the relative error $\eta$, at the cost of a quadratic dependence on $r$.
We believe this quadratic dependence is a limitation of the proof technique, rather than a feature of the algorithm.

In the special case that $\rank(\mat{A}) \leq r$,
our analysis (\Cref{lem:single_step}) ensures that
$\mat{A} = \mat{\hat{A}}^{(k)}$ for any $k \geq r$.
We note that both greedy and DPP sampling achieve a perfect approximation quality when $k \geq \rank(\mat{A})$, but the same is not true for uniform sampling or for RLS sampling.

Due to the close relationship between \RPCholesky and randomly pivoted QR (\Cref{sec:close_relationship}), our analysis also leads to the following new error bound:
\begin{corollary}[Randomly pivoted QR] \label{cor:main_bound}
Fix a matrix $\mat{B} \in \complex^{M \times N}$, a target approximation rank $r$, and a tolerance $\varepsilon > 0$.
The rank-$k$ column projection approximation $\mat{\hat{B}}^{(k)}$ produced by
$k$ steps of randomly pivoted QR (\Cref{alg:adaptive} with block size $T = 1$)
attains the bound
\begin{equation*}
    \expect \lVert \mat{B} - \mat{\hat{B}}^{(k)} \rVert^2_{\rm F}
    \leq (1 + \varepsilon) \cdot \lVert \mat{B} - \lowrank{\mat{B}}_r \rVert^2_{\rm F},
\end{equation*}
provided that the number $k$ of columns satisfies 
\begin{equation*} \label{eq:qr-thm-k-bd}
    k \geq \frac{r}{\varepsilon} + \min\biggl\{ r \log\biggl(\frac{1}{\varepsilon \eta}\biggr), r + r \log_+ \biggl(\frac{2^r}{\varepsilon}\biggr)\biggr\}.
\end{equation*}
The relative error $\eta$ is defined by $\eta \coloneqq \lVert \mat{B} - \lowrank{\mat{B}}_r \rVert^2_{\rm F} / \lVert \mat{B} \rVert^2_{\rm F}$.
\end{corollary}
In forthcoming work, we will extend 
these results
to address 
block randomly pivoted QR (\Cref{alg:adaptive}) and
block \RPCholesky (\Cref{alg:rpcholesky_blocking})
with a fixed block size $T > 1$. This analysis is more
complicated and requires several additional ideas.

\subsection{\texorpdfstring{Proof of \Cref{thm:main_bound}}{Proof of Theorem 3.1}}
\label{sec:proof_1}

Let $\mat{A}$ be a psd input matrix.  Recall the definition~\eqref{eqn:chol-updates}
of the approximation $\mat{\hat{A}}^{(i)}$ and residual $\mat{A}^{(i)}$ matrices
generated by the pivoted partial Cholesky algorithm.  Recall that \RPCholesky samples
each pivot from the distribution \eqref{eq:sampling_dist}.

The proof of \Cref{thm:main_bound} is based on the properties of the \emph{expected residual function}:
\begin{equation} \label{eq:Phi}
  \mat{\Phi}(\mat{A}) \coloneqq \expect \bigl[\mat{A}^{(1)}\, \big| \, \mat{A}\bigr] 
  = \mat{A} - \frac{\mat{A}^2}{\tr\mat{A}}.
\end{equation}
This function returns the expectation of the residual $\mat{A}^{(1)}$ after applying one step of \RPCholesky to the psd matrix $\mat{A}$.
The equality \eqref{eq:Phi} follows from a short computation using the sampling distribution~\eqref{eq:sampling_dist} and the definition~\eqref{eqn:chol-updates} of the residual.
Note that $\mat{\Phi}$ is defined on psd matrices of any dimension.

The first lemma describes some basic facts about the expected residual function $\mat{\Phi}$.
We postpone the proof to \Cref{sec:concave_proof}.

\begin{lemma}[Expected residual] \label{lem:concave}
    The expected residual map $\mat{\Phi}$ defined in~\eqref{eq:Phi} is positive, monotone, and concave with respect to the psd order. That is, for all psd $\mat{A},\mat{H}$ with the same dimensions,
    \begin{gather}
    \mathbf{0} \preceq \mat{\Phi}(\mat{A}) \preceq \mat{\Phi}(\mat{A} + \mat{H}); \label{eq:monotonicity} \\
    \theta \mat{\Phi}(\mat{A}) + (1-\theta)\mat{\Phi}(\mat{H})
    \preceq \mat{\Phi}(\theta \mat{A} + (1-\theta) \mat{H})
    \quad\text{for all $\theta \in [0,1]$}. \label{eq:concavity}
    \end{gather}
\end{lemma}

The second lemma describes how the trace of the residual declines after multiple steps
of the \RPCholesky procedure.  This is the key new ingredient in our argument.  The proof appears in \Cref{sec:contraction_proof}.

\begin{lemma}[Contraction rate] \label{lem:trace_bound}
Consider the $k$-fold composition of the expected residual~\eqref{eq:Phi}:
\begin{equation*}
    \mat{\Phi}^{\circ k} \coloneqq \underbrace{\mat{\Phi} \circ \mat{\Phi} \circ \cdots \circ \mat{\Phi}}_{\textnormal{$k$ times}}
    \quad\text{for each $k \in \mathbb{N}$.}
\end{equation*}
Fix $r \in \mathbb{N}$.
For each psd matrix $\mat{H}$ and each $\Delta > 0$,
\begin{equation*}
    \tr \mat{\Phi}^{\circ k}(\mat{H}) \le \tr (\mat{H} - \lowrank{\mat{H}}_r) + \Delta \tr \mat{H} \quad \text{when} \quad
    k \ge \frac{r \tr \bigl(\mat{H} - \lowrank{\mat{H}}_r\bigr)}{\Delta \tr \mat{H}} + r \log_+ \biggl( \frac{1}{\Delta} \biggr).
\end{equation*}
\end{lemma}

Last, we present a bound which shows that the error after $k$ steps of \RPCholesky
is comparable with the error in the best rank-$k$ approximation.
This lemma improves on an earlier result~\cite[Prop.~2]{DV06} by reducing a $(k+1)!$ factor to $2^k$, but the proof is entirely different.
This bound is responsible for the final term in \cref{eq:alternative_bound}, and the proof appears in \Cref{sec:doubling_proof}.

\begin{lemma}[Error doubling] \label{lem:single_step}
  For each psd matrix $\mat{A}$, the residual matrix $\mat{A}^{(k)}$ after applying $k$ steps of \RPCholesky satisfies
  \begin{equation*} \label{eq:repeated_doubling}
    \expect \tr \mat{A}^{(k)} \le 2^k \tr(\mat{A} - \lowrank{\mat{A}}_k)
    \quad\text{for each $k \in \mathbb{N}$.}
  \end{equation*}
\end{lemma}

With these results at hand, we quickly establish the main error bound for \RPCholesky.

\begin{proof}[Proof of \Cref{thm:main_bound}]
  Fix a psd matrix $\mat{A}$.  By the concavity~\eqref{eq:concavity} of the expected residual map~\eqref{eq:Phi} and a matrix version of Jensen's inequality~\cite[Thm.~4.16]{Car09}, the residual matrices satisfy
  \begin{equation*}
    \expect \mat{A}^{(j)} = \expect\bigl[\expect \bigl[ \mat{A}^{(j)} \,\big|\, \mat{A}^{(j-1)} \bigr]\bigr] = \expect \mat{\Phi}(\mat{A}^{(j-1)}) \preceq \mat{\Phi}(\expect \mat{A}^{(j-1)})
    \quad\text{for each $j \in \{1, \dots, k\}$.}
  \end{equation*}
  Next, by monotonicity~\eqref{eq:monotonicity} of $\mat{\Phi}$ and the last display,
  \begin{equation*}
  \label{eq:intermediate}
      \expect \mat{A}^{(k)} 
      \preceq \mat{\Phi}(\expect \mat{A}^{(k-1)})
      \preceq \mat{\Phi} \circ \mat{\Phi}(\expect \mat{A}^{(k-2)})
      \preceq \cdots \preceq \mat{\Phi}^{\circ k}(\mat{A}).
  \end{equation*}
  Using the fact that the trace is linear and preserves the psd order,
  \begin{equation*}
      \expect \tr\bigl(\mat{A} - \mat{\hat{A}}^{(k)} \bigr) =
      \expect \tr \mat{A}^{(k)} \leq \tr \mat{\Phi}^{\circ k}(\mat{A}).
  \end{equation*}
  Next, we apply \Cref{lem:trace_bound} with $\mat{H} = \mat{A}$ and $\Delta = \varepsilon \eta$ to see that
  \begin{equation*}
  \expect \tr\bigl(\mat{A} - \mat{\hat{A}}^{(k)}\bigr) 
  \leq (1 + \varepsilon) \cdot \tr (\mat{A} - \lowrank{\mat{A}}_r)
  \quad\text{when}\quad
  k \geq \frac{r}{\varepsilon} + r \log_+\biggl(\frac{1}{\varepsilon \eta}\biggr),
  \end{equation*}
  where we have recognized the relative error $\eta$.
  Last, we can replace $\log_+$ with $\log$ because the statement holds trivially for any Nystr\"om approximation with any number of columns if $\varepsilon \eta \geq 1$.
  
  By a similar argument,
  \begin{equation*}
  \label{eq:k-r}
    \expect \tr\bigl(\mat{A} - \mat{\hat{A}}^{(k)}\bigr) = \expect \tr \mat{A}^{(k)} \leq \tr \mat{\Phi}^{\circ (k - r)}\bigl(\expect \mat{A}^{(r)}\bigr).
  \end{equation*}
  We can apply \Cref{lem:trace_bound} with $\mat{H} = \expect \mat{A}^{(r)}$
  and $\Delta = \varepsilon \eta \cdot \tr \mat{A} \slash \expect \tr \mat{A}^{(r)}$
  to see that
  \begin{equation}
  \label{eq:combine1}
      \expect \tr\bigl(\mat{A} - \mat{\hat{A}}^{(k)}\bigr) \leq \tr\bigl( \expect \mat{A}^{(r)} - \lowrank{ \expect \mat{A}^{(r)} }_r \bigr)
      + \varepsilon \tr (\mat{A} - \lowrank{\mat{A}}_r)
   \end{equation}
   provided the number of columns satisfies
   \begin{equation}
   \label{eq:combine2}
      k-r \geq \frac{r \tr\bigl( \expect \mat{A}^{(r)} - \lowrank{ \expect \mat{A}^{(r)} }_r \bigr)}
      {\varepsilon \tr (\mat{A} - \lowrank{\mat{A}}_r)} + r \log_+\biggl(\frac{\expect \tr \mat{A}^{(r)}}{\varepsilon \tr (\mat{A} - \lowrank{\mat{A}}_r)}\biggr).
  \end{equation}
  This bound can be simplified as follows.
  Observe that the (random) residual matrix $\mat{A}^{(r)}$ is a Schur complement of $\mat{A}$, so it satisfies $\expect \mat{A}^{(r)} \preceq \mat{A}$.
  By the Ky Fan variational principle~\cite[Thm.~8.17]{Zha11}, the best rank-$r$ approximation error in the trace norm is monotone with respect to the psd order, so
  \begin{equation*}
  \tr\bigl( \expect \mat{A}^{(r)} - \lowrank{ \expect \mat{A}^{(r)} }_r \bigr)
    \leq \tr( \mat{A} - \lowrank{ \mat{A} }_r ).
  \end{equation*}
  Additionally, \Cref{lem:single_step} guarantees that $\expect \tr \mat{A}^{(r)} \leq 2^r \tr( \mat{A} - \lowrank{\mat{A}}_r)$.
  Using these facts,
  we can simplify \eqref{eq:combine1}--\eqref{eq:combine2} to show that
  \begin{equation*}
  \expect \tr\bigl(\mat{A} - \mat{\hat{A}}^{(k)}\bigr) \leq 
  (1 + \varepsilon) \cdot \tr (\mat{A} - \lowrank{\mat{A}}_r)
  \quad\text{when}\quad
  k \geq r + \frac{r}{\varepsilon} + r \log_+\biggl(\frac{2^r}{\varepsilon}\biggr),
  \end{equation*}
  This completes the proof of the second bound.
\end{proof}

\subsection{\texorpdfstring{Proof of \Cref{lem:concave}}{Proof of Lemma 3.2}} \label{sec:concave_proof}
    Let $\mat{A}, \mat{H}$ be psd matrices, and recall the definition~\eqref{eq:Phi} of the expected residual map $\mat{\Phi}$.
    First, to prove that $\mat{\Phi}$ is positive, note that 
  \begin{equation*}
  \mat{\Phi}(\mat{H}) = \biggl( \Id - \frac{\mat{H}}{\tr \mat{H}} \biggr)\mat{H}  \succeq \mathbf{0}. 
  \end{equation*}
  Next, to establish concavity, we choose $\theta \in [0, 1]$, set $\overline{\theta} \coloneqq 1-\theta$, and make the calculation
  \begin{equation*}
    \mat{\Phi}(\theta \mat{A} + \overline{\theta} \mat{H})
    - \theta \mat{\Phi}(\mat{A})
    - \overline{\theta} \mat{\Phi}(\mat{H}) \\
    = \frac{\theta\overline{\theta}}{\theta \tr \mat{A} + \overline{\theta} \tr \mat{H}}
    \biggl( \sqrt{\frac{\tr \mat{H}}{\tr \mat{A}}} \mat{A} - \sqrt{\frac{\tr \mat{A}}{\tr \mat{H}}} \mat{H} \biggr)^2 \succeq \mathbf{0}.
  \end{equation*}
  Last, to establish monotonicity, observe that $\mat{\Phi}$ is positive homogeneous; that is, $\mat{\Phi}(\tau \mat{A}) = \tau \mat{\Phi}(\mat{A})$ for $\tau \ge 0$.
  Invoking the concavity property~\eqref{eq:concavity} with $\theta = 1/2$,
  \begin{equation*}
    \mat{\Phi}(\mat{A} + \mat{H}) = 2 \, \mat{\Phi} \biggl( \frac{\mat{A}+\mat{H}}{2} \biggr) \succeq \mat{\Phi}(\mat{A}) + \mat{\Phi}(\mat{H}) \succeq \mat{\Phi}(\mat{A}).
  \end{equation*}
  We have used the positivity of $\mat{\Phi}(\mat{H})$ in the last step. \hfill $\proofbox$

\subsection{\texorpdfstring{Proof of \Cref{lem:trace_bound}}{Proof of Lemma 3.3}}
\label{sec:contraction_proof}

We break the proof into several steps.

\subsubsection{Step 1: Reduction to diagonal case}

    First, we show that it suffices to consider the case of a diagonal matrix.  Let $\mat{H}$ be an $N \times N$ psd matrix
    with eigendecomposition $\mat{H} = \mat{V}\mat{\Lambda}\mat{V}^*$.  The definition~\eqref{eq:Phi} of the expected residual map implies that
    \begin{equation*}
      \mat{\Phi}(\mat{H}) = \mat{V} \mat{\Phi}(\mat{\Lambda}) \mat{V}^*.
    \end{equation*}
    By iteration, the same relation holds with $\mat{\Phi}^{\circ k}$ in place of $\mat{\Phi}$.
    In particular, $\tr \mat{\Phi}^{\circ k}(\mat{H}) = \tr \mat{\Phi}^{\circ k}(\mat{\Lambda})$.
    Therefore, we may restrict our attention to the diagonal case where $\mat{H} = \mat{\Lambda}$.

\subsubsection{Step 2: Identification of worst-case matrix}

    Second, we obtain an upper bound on $\tr \mat{\Phi}^{\circ k}(\mat{\Lambda})$.
    We accomplish this goal by identifying the worst-case set of eigenvalues.
    Because the map $\mat{\Lambda} \mapsto \tr \mat{\Phi}^{\circ k}(\mat{\Lambda})$
    is concave and invariant under permutations of its arguments,
    averaging together some of the eigenvalues $\lambda_1, \ldots, \lambda_N$ of $\mat{\Lambda}$ can only increase the value of $\tr \mat{\Phi}^{\circ k}(\mat{\Lambda})$.
    Therefore, by introducing the function
    \begin{equation*}
      f_k(a,b,r,q) \coloneqq \tr \mat{\Phi}^{\circ k}\biggl( \diag\biggl(\,\underbrace{\frac{a}{r},\ldots,\frac{a}{r}}_{\textnormal{$r$ times}},\underbrace{\frac{b}{q},\ldots,\frac{b}{q}}_{\textnormal{$q$ times}}\,\biggr) \biggr),
    \end{equation*}
    we obtain the upper bound
    \begin{equation}
    \label{eq:combine_with}
      \tr \mat{\Phi}^{\circ k}(\mat{\Lambda}) 
      \leq f_k( \tr \lowrank{\mat{\Lambda}}_r,
      \tr ( \mat{\Lambda} - \lowrank{\mat{\Lambda}}_r), r, N - r).
    \end{equation}
    For further reference, we note that $f_k(a, b, r, q)$ is weakly increasing as a function of $q$.
    Indeed, for every tuple $(a, b, r, q)$, 
    \begin{align*}
        f_k(a, b, r, q)
        &= \tr \mat{\Phi}^{\circ k}\biggl( \diag\biggl(\,\underbrace{\frac{a}{r},\ldots,\frac{a}{r}}_{\textnormal{$r$ times}},\underbrace{\frac{b}{q},\ldots,\frac{b}{q}}_{\textnormal{$q$ times}}\,\biggr)\biggr) = \tr \mat{\Phi}^{\circ k}\biggl( \diag\biggl(\,\underbrace{\frac{a}{r},\ldots,\frac{a}{r}}_{\textnormal{$r$ times}},\underbrace{\frac{b}{q},\ldots,\frac{b}{q}}_{\textnormal{$q$ times}},0\biggr)\biggr) \\
          &\le \tr \mat{\Phi}^{\circ k}\biggl( \diag\biggl(\,\underbrace{\frac{a}{r},\ldots,\frac{a}{r}}_{\textnormal{$r$ times}},\underbrace{\frac{b}{q+1},\ldots,\frac{b}{q+1}}_{\textnormal{$q+1$ times}}\,\biggr)\biggr) = f_k(a,b,r,q+1).
    \end{align*}
    We have exploited the fact that $\mat{\Phi}$ is defined for matrices of every dimension.

\subsubsection{Step 3: Dynamics of the error}

    Next, we derive a worst-case expression for the error $f_k(a, b, r, q)$.
    We accomplish this task by identifying a discrete-time dynamical system
    that describes the evolution of the residuals.
    For each $k = 0, 1, 2, \ldots$,
    define the nonnegative quantities $a^{(k)}$ and $b^{(k)}$ via the relation
    \begin{equation*}
      \mat{\Phi}^{\circ k} \biggl( \diag \biggl(\, \underbrace{\frac{a}{r},\ldots,\frac{a}{r}}_{\text{$r$ times}}, \underbrace{\frac{b}{q},\ldots,\frac{b}{q}}_{\text{$q$ times}}\,\biggr) \biggr)
      \eqqcolon \diag \biggl(\, \underbrace{\frac{a^{(k)}}{r},\ldots,\frac{a^{(k)}}{r}}_{\text{$r$ times}}, \underbrace{\frac{b ^{(k)}}{q},\ldots,\frac{b^{(k)}}{q}}_{\text{$q$ times}} \,\biggr).
    \end{equation*}
    By the definition~\eqref{eq:Phi} of the expected residual map $\mat{\Phi}$, the quantities $a^{(k)}$ and $b^{(k)}$ satisfy the recurrence relations
    \begin{equation*}
      a^{(k)} - a^{(k-1)} = -\frac{ \bigl( a^{(k-1)} \bigr)^2 }{r\bigl(a^{(k-1)} + b^{(k-1)}\bigr)}
      \quad\text{and}\quad
      b^{(k)} - b^{(k-1)} = -\frac{ \bigl( b^{(k-1)} \bigr)^2 }{q\bigl(a^{(k-1)} + b^{(k-1)}\bigr)}
    \end{equation*}
    with initial conditions $a^{(0)} = a$ and $b^{(0)} = b$.
    This construction guarantees $f_k(a, b, r, q) = a^{(k)} + b^{(k)}$.
    Additionally, the quantities $a^{(k)}$ and $b^{(k)}$ converge as $q \rightarrow \infty$
    to limiting values $\overline{a}^{(k)}$ and $\overline{b}^{(k)} \equiv b$, where the sequence $\overline{a}^{(k)}$ satisfies
    \begin{equation*}
      \overline{a}^{(k)} - \overline{a}^{(k-1)} = -\frac{ \bigl( \overline{a}^{(k-1)} \bigr)^2 }{r\bigl(\overline{a}^{(k-1)} + b\bigr)}
      \quad\text{with initial condition $\overline{a}^{(0)} = a$.}
    \end{equation*}
    It follows that 
    \begin{equation*}
        f_k(a, b, r, q)
        \leq \overline{a}^{(k)} + b \leq a + b.
    \end{equation*}
    We have used the facts that $f_{k}(a,b,r,q)$ is increasing in $q$ and $\overline{a}^{(k)}$ is decreasing in $k$.

\subsubsection{Step 4: Comparison with continuous-time dynamics}

    All that remains is to determine how quickly $\overline{a}^{(k)}$ decreases as a function of $k$.
    To that end, we pass from discrete time to continuous time.
    At each instant $t = 0, 1, 2, \ldots$,
    the discrete-time process $\overline{a}^{(t)}$ is bounded from above
    by the continuous-time process $x(t)$ satisfying the ODE
    \begin{equation*}
      \frac{\diff}{\diff t} x(t) = -\frac{x(t)^2}{r(x(t) + b)} \quad\text{with initial condition $x(0) = a$.}
    \end{equation*}
    The comparison between discrete- and continuous-time processes holds because $x \mapsto -x^2 \slash (rx + rb)$ is decreasing over the range $x \in (0, \infty)$.
    Next, assuming $\Delta (a + b) \leq a$, we can use separation of variables to solve for the time $t_{\star}$ at which $x(t_{\star}) = \Delta (a + b)$:
    \begin{multline*}
      t_{\star} = 
      \int_{x=a}^{\Delta (a + b)}
      -\frac{r(x + b)}{x^2} \,\diff x
      = rb \biggl( \frac{1}{\Delta (a + b)} - \frac{1}{a}\biggr) + r \log \biggl( \frac{a}{\Delta (a + b)} \biggr)
      \leq \frac{rb}{\Delta( a + b)} 
      + r \log\biggl( \frac{1}{\Delta} \biggr).
    \end{multline*}
    It follows (even for $\Delta (a + b) > a$) that
    \begin{equation*}
      \overline{a}^{(k)} \leq \Delta (a + b)
      \quad \textnormal{when} \quad
    k \geq
    \frac{rb}{\Delta (a + b)} + r \log_+\biggl( \frac{1}{\Delta} \biggr).
    \end{equation*}
    To conclude, substitute $a = \tr \lowrank{\mat{\Lambda}}_r$ and $b = \tr(\mat{\Lambda} - \lowrank{\mat{\Lambda}}_r)$ and combine with \eqref{eq:combine_with}.
    \hfill $\proofbox$
\subsection{\texorpdfstring{Proof of \Cref{lem:single_step}}{Proof of Lemma 3.1}} \label{sec:doubling_proof}

Let $\mat{P}$ denote the orthogonal projection onto the $k$ dominant eigenvectors of $\mat{A}$ and set $\mat{P}_\perp \coloneqq \Id - \mat{P}$.
Apply the Ky Fan variational principle \cite[Thm.~8.17]{Zha11} to write
\begin{equation*}
\tr\bigl(\mat{A}^{(1)} - \lowrank{\mat{A}^{(1)}}_{k-1}\bigr) 
= \min \Biggl\{ \sum_{j=k}^N \vec{u}_j^* \mat{A}^{(1)} \vec{u}_j \,:\, \vec{u}_k,\ldots,\vec{u}_N \textrm{ orthonormal} \Biggr\}.
\end{equation*}
Choose the vectors $\vec{u}_{k+1},\ldots,\vec{u}_N$ to be unit-norm eigenvectors $\vec{v}_{k+1}(\mat{A}),\ldots,\vec{v}_N(\mat{A})$ associated with the smallest eigenvalues.
To choose the last vector $\vec{u}_k$, let us separately consider the cases $\mat{P}(s_1, s_1) = 0$ and $\mat{P}(s_1, s_1) > 0$, where $\mat{P}(i,i)$ denotes the $(i, i)$-entry of $\bm{P}$.

On the event $\mat{P}(s_1,s_1) = 0$, choose $\vec{u}_k \coloneqq \vec{v}_1(\mat{A})$ and apply the crude bound
\begin{equation} \label{eq:combo1}
\begin{split}
\tr\bigl( \mat{A}^{(1)} - \lowrank{\mat{A}^{(1)}}_{k-1} \bigr)
&\leq \sum_{j = k}^{N} \vec{u}_j^* \mat{A}^{(1)} \vec{u}_j \\
&\leq \sum_{j = k}^{N} \vec{u}_j^* \mat{A} \vec{u}_j
=
\tr(\mat{A} - \lowrank{\mat{A}}_{k}) + \lambda_1(\mat{A}).
\end{split}
\end{equation}
This relation holds because $\mat{A}^{(1)} \preceq \mat{A}$ and $\sum_{j=k+1}^N \lambda_j(\mat{A}) = \tr(\mat{A} - \lowrank{\mat{A}}_{k})$.

On the event $\mat{P}(s_1,s_1) > 0$, 
choose $\vec{u}_k \coloneqq \mat{P} \mathbf{e}_{s_1} / \lVert \smash{\mat{P}\mathbf{e}_{s_1}} \rVert$
, and observe that $\vec{u}_k$ is orthonormal to the other vectors by the choice of $\mat{P}$.
This gives the bound
\begin{equation} \label{eq:towards_good}
\begin{split}
\tr\bigl(\mat{A}^{(1)} - \lowrank{\mat{A}^{(1)}}_{k-1}\bigr) 
&\leq \sum_{j=k+1}^N \vec{u}_j^* \mat{A}^{(1)} \vec{u}_j + \frac{\mathbf{e}_{s_1}^* \mat{P}\mat{A}^{(1)}\mat{P}\mathbf{e}_{s_1}}{\mathbf{e}_{s_1}^*\mat{P}\mathbf{e}_{s_1}} \\
&\leq \tr(\mat{A} - \lowrank{\mat{A}}_k) + \frac{(\mat{P}\mat{A}^{(1)}\mat{P})(s_1,s_1)}{\mat{P}(s_1,s_1)}.
\end{split}
\end{equation}
Using the facts that $\mat{A}^{(1)} = \mat{A} - \mat{A} \mathbf{e}_{s_1}\mathbf{e}_{s_1}^* \mat{A} / \mat{A}(s_1,s_1)$ and $\mat{PA} = \mat{AP}$, it follows
\begin{equation*}
(\mat{P}\mat{A}^{(1)}\mat{P})(s_1,s_1) = (\mat{A}\mat{P})(s_1,s_1) - \frac{((\mat{A}\mat{P})(s_1,s_1))^2}{\mat{A}(s_1,s_1)}
= \frac{(\mat{A}\mat{P})(s_1,s_1)(\mat{A}\mat{P}_\perp)(s_1,s_1)}{\mat{A}(s_1,s_1)}.
\end{equation*}
Combine the bounds \eqref{eq:combo1} and \eqref{eq:towards_good} and sum over the selection probabilities
$\prob\{s_1 = i\} = \mat{A}(i,i) / \tr \mat{A}$ to evaluate
\begin{equation} \label{eq:combo2}
\begin{split}
\expect \bigl[ \tr\bigl(\mat{A}^{(1)} - \lowrank{\mat{A}^{(1)}}_{k-1}\bigr) \bigr]
&\le \tr(\mat{A} - \lowrank{\mat{A}}_k)
+ \sum_{\mat{P}(i,i) = 0}
\frac{\mat{A}(i,i)}{\tr \mat{A}} \cdot
\lambda_1(\mat{A})
+ \sum_{\mat{P}(i,i) > 0}
\frac{\mat{A}(i,i)}{\tr \mat{A}} \cdot \frac{(\mat{A}\mat{P})(i,i)(\mat{A}\mat{P}_\perp)(i,i)}{\mat{A}(i,i) \mat{P}(i,i)} \\
&\le \tr(\mat{A} - \lowrank{\mat{A}}_k)
+ \sum_{i=1}^N \frac{(\mat{A}\mat{P}_\perp)(i,i)}{\tr \mat{A}} \cdot \lambda_1(\mat{A}) \\
&= \biggl( 1 + \frac{\lambda_1(\mat{A})}{\tr \mat{A}} \biggr) \tr(\mat{A} - \lowrank{\mat{A}}_k).
\end{split}
\end{equation}
The middle line uses the inequalities
\begin{equation*}
    \mat{P}\mathbf{e}_i = \vec{0}
    \implies
    \mat{A}(i,i) = (\mat{A} \mat{P})(i,i) + (\mat{A} \mat{P}_{\perp})(i,i)
    = (\mat{A} \mat{P}_{\perp})(i,i).
\end{equation*}
and $(\mat{A}\mat{P})(i,i) \le \lambda_1(\mat{A}) \mat{P}(i,i)$.

Since $\lambda_1(\mat{A}) \le \tr \mat{A}$, the bound  \eqref{eq:combo2} can be weakened to give
\begin{equation*}
\expect \tr\bigl(\mat{A}^{(1)} - \lowrank{\mat{A}^{(1)}}_{k-1}\bigr)
\leq 2 \tr(\mat{A} - \lowrank{\mat{A}}_{k}).
\end{equation*}
By iterating, we conclude that
\begin{equation*}
\expect \tr \mat{A}^{(k)} 
\le 2 \expect \tr\bigl(\mat{A}^{(k-1)} - \lowrank{\mat{A}^{(k-1)}}_1 \bigr)
\le 4 \expect \tr\bigl(\mat{A}^{(k-2)} - \lowrank{\mat{A}^{(k-2)}}_2 \bigr) \\
\le \cdots \leq
2^k \tr(\mat{A} - \lowrank{\mat{A}}_k).
\end{equation*}
This estimate is the statement of the doubling bound. \hfill $\proofbox$

\section{Conclusion} \label{sec:conclusion}

This work has demonstrated the utility of \RPCholesky for low-rank approximation of a psd matrix $\mat{A} \in \mathbb{C}^{N \times N}$.
\RPCholesky 
allows us to accelerate many kernel algorithms, such as kernel ridge regression and kernel spectral clustering.
\RPCholesky reduces the computational cost of these kernel methods from $\order(N^3)$ operations to just
$\order(k^2 N)$ operations,
where the approximation rank $k$ can be much smaller than the matrix dimension $N$.

Numerical experiments suggest that \RPCholesky improves over other column Nystr\"om approximation algorithms in terms of floating-point operations and memory footprint.
Given a fixed approximation rank $k$, \RPCholesky requires a very small number of entry evaluations, just $(k + 1) N$.
\RPCholesky typically produces approximations that match or improve on the greedy method, uniform sampling, RLS sampling, and DPP sampling.
Moreover, theoretical error bounds guarantee that \RPCholesky converges nearly as fast as possible in the expected trace norm.

Taken as a whole, this work paves the way for greater use of \RPCholesky in the future.
Additionally, we believe \RPCholesky can be pushed even further, for example, through combinations with accelerated methods for prediction, clustering, and other learning tasks.
The future is indeed bright for this simple yet surprisingly effective algorithm.

\section*{Acknowledgements}

We thank Mateo D{\'i}az, Zachary Frangella, Marc Gilles, Eitan Levin, Eliza O'Reilly, Jonathan Weare, and Aaron Dinner for helpful discussions and corrections.

\appendix

\section{Details of numerical experiments} \label{app:numerical_details}
The numerical experiments in \Cref{sec:compare,sec:rpc} are implemented in Python with code available at \url{https://github.com/eepperly/Randomly-Pivoted-Cholesky}.
Below we provide further implementation details for \RPCholesky, DPP sampling, RLS sampling, and greedy pivoting:

\textbf{\RPCholesky}. All the \RPCholesky results in \Cref{sec:compare,sec:rpc} use the simple \RPCholesky method (\Cref{alg:rpcholesky}) with block size $T = 1$.

\textbf{DPP sampling}. We use the DPP samplers from the DPPy Python package \cite{GPBV19}.
However, the samplers from this package all produce error messages when tried on certain inputs, particularly when $k$ is large and $\mat{A}$ is nearly low-rank.
The theoretically fastest \texttt{alpha} sampler cannot be used for our numerical comparisons since it has error messages for both the \textbf{Smile} and \textbf{Spiral} examples.
For this reason, we used the \texttt{vfx} sampler \cite{DCV19} when assessing the entry access cost of DPP sampling, and we produced \Cref{fig:comparison} using the comparatively slow \texttt{GS} sampler which requires a full eigendecomposition of $\mat{A}$.
All these DPP samplers generate exact samples from the $k$-DPP distribution \eqref{eq:exact_dpp}; we did not test with inexact samplers based on MCMC \cite{AGR16}.

\textbf{RLS sampling}.
Several RLS sampling algorithms have been introduced in the literature,
including recursive ridge leverage scores (RRLS) \cite{MM17}, SQUEAK \cite{CLV17}, and BLESS \cite{RCCR18}.
The analysis and experiments in this paper are based on the RRLS algorithm \cite{MM17}, which is implemented in the \texttt{recursiveNystrom} method \cite{VMM19} for Python.
For comparison, 
we have also implemented
RLS sampling by directly forming the matrix
$\mat{A} (\mat{A} + \lambda \Id)^{-1}$ and sampling $k$ indices with probability weighted defined by the ridge leverage scores \eqref{eq:rls}.
Even with the ridge leverage scores computed exactly in this way, the relative trace-norm error for the \textbf{Smile} example with $k = 100$ is still roughly $10^{-2}$, five orders of magnitude larger than the error due to \RPCholesky.
This indicates that the poor performance of RLS on this example is not due to errors in the approximation of the ridge leverage scores.

\textbf{Greedy pivoting}.
Following the original release of this work as an arXiv preprint, S.\ Steinerberger \cite{Ste24} pointed out that the greedy pivoting method exhibits degenerate behavior if many diagonal entries of the residual matrix are tied for the largest value, up to floating point errors.
In case of a tie, existing implementations of the greedy algorithm select columns according to the default ordering, always choosing the first column with the maximal diagonal value.
The impact of the ordering can be removed by randomly pre-shuffling the data $\{\vec{x}_1,\ldots,\vec{x}_N\}$, which we did for the \textbf{Smile} and \textbf{Spiral} data sets in \cref{fig:comparison}.
Pre-shuffling slightly improved the performance of greedy pivoting for the \textbf{Spiral} data.
We note that existing implementations of greedy pivoting, such as LAPACK's \texttt{pstrf}, do not perform this shuffling and process the data in the given order.

\section{Pseudocode for QR methods}
Pseucode for column-pivoted partial QR and block randomly pivoted QR are provided as \Cref{alg:pivoted-qr} and \Cref{alg:adaptive}.

\begin{algorithm}[t]
  \caption{Column-pivoted partial QR \label{alg:pivoted-qr}}
  \textbf{Input:} Matrix $\mat{B} \in \mathbb{C}^{M \times N}$; approximation rank $k$
  
  \textbf{Output:} Pivot set $\set{S} = \{s_1, \dots, s_k\}$; matrices $\mat{Q} \in \mathbb{C}^{M \times k}$ and $\mat{R} \in \mathbb{C}^{k \times N}$ defining approximation $\mat{\hat{B}} = \mat{Q} \mat{R}$
  \begin{algorithmic}
    \State Initialize $\mat{Q} \leftarrow \mat{0}_{M \times k}$ and $\mat{R} \leftarrow \mat{0}_{k \times N}$
    \For{$i = 1$ to $k$}
    \State Select a pivot column $s_i \in \{1, \dots, N\}$
    \Comment{See \Cref{sec:qr-pivot-rules}}
    \State $\vec{g} \leftarrow \mat{B}(:, s_i)$
    \Comment{Extract $s_i$ column of input matrix}
    \State $\vec{g} \leftarrow \vec{g} - \mat{Q}(:, 1:i-1) \mat{R}(1:i-1, s_i)$
    \Comment{Remove projection on previously chosen columns}
    \State $\vec{g} \leftarrow \vec{g} / \lVert \vec{g} \rVert$
    \Comment{Normalize}
    \State $\mat{Q}(:, i) \leftarrow \vec{g}$
    \Comment{Update approximation}
    \State $\mat{R}(i, :) \leftarrow \vec{g}^* \mat{B}$
    \EndFor
  \end{algorithmic}
\end{algorithm}

\begin{algorithm}[t]
  \caption{Block randomly pivoted partial QR (aka adaptive sampling~\cite{DRVW06}) 
  \label{alg:adaptive}}
  \textbf{Input:} Matrix $\mat{B} \in \mathbb{C}^{M \times N}$; block size $T$; approximation rank $k$ which is a multiple of $T$
  
  \textbf{Output:} Pivot set $\set{S}$; matrices $\mat{Q}$ and $\mat{R}$ defining approximation $\mat{\hat{B}} = \mat{Q} \mat{R}$
  \begin{algorithmic}
    \State Initialize $\mat{Q} \leftarrow \mat{0}_{M \times k}$, $\mat{R} \leftarrow \mat{0}_{k \times N}$, and $\set{S} \leftarrow \emptyset$
    \State Initialize $\vec{p} \leftarrow \Call{SquaredColumnNorms}{\mat{B}}$
    \For{$i = 0$ to $k/T - 1$}
    \State 
    Sample $s_{iT + 1}, \ldots, s_{iT + T} \stackrel{\rm iid}{\sim} \vec{p} / \sum_{j=1}^N \vec{p}(j)$
    \Comment{Probability prop.~to squared column norms of residual}
    \State $\set{S}' \leftarrow \Call{Unique}{\{s_{iT + 1}, \ldots, s_{iT + T}\}}$
    \State $\set{S} \leftarrow \set{S} \cup \set{S}'$
    \State $\mat{G} \leftarrow \mat{B}(:, \set{S}^{'})$
    \Comment{Evaluate columns $\set{S}'$ of input matrix}
    \State $\mat{G} \leftarrow \mat{G} - \mat{Q} \mat{R}(:, \set{S}')$
    \Comment{Remove projections on previously chosen columns}
    \State $\mat{G} \leftarrow \Call{Orth}{\mat{G}}$
    \Comment{Orthonormalize}
    \State $\mat{Q}(:, (iT + 1):(iT + |\set{S}'|)) \leftarrow \mat{G}$
    \Comment{Update approximation}
    \State $\mat{R}((iT + 1):(iT + |\set{S}'|), :) \leftarrow \mat{G}^* \mat{B}$
    \State $\vec{p} \leftarrow \vec{p} - \Call{SquaredColumnNorms}{\mat{G}^* \mat{B}}$
    \Comment{Track squared column norms of residual}
    \State $\vec{p}\leftarrow \max\{\vec{p},\vec{0}\}$ \Comment{Ensure $\vec{p}$ remains nonnegative}
    \EndFor
  \State Remove zero columns from $\mat{Q}$ and zero rows from $\mat{R}$
  \end{algorithmic}
\end{algorithm}

\section{Comparison with other methods}
\label{sec:related-work}
In this section, we give an example which shows that any Nystr\"om method needs at least $r/\varepsilon$ columns to guarantee a $(r, \varepsilon)$-approximation (\Cref{sec:lower_bound}).
Then, we analyze how many columns are needed to obtain a $(r, \varepsilon)$-approximation using the greedy method (\Cref{sec:greedy}), uniform sampling (\Cref{sec:uniform}), DPP sampling (\Cref{sec:dpp}), and RLS sampling (\Cref{sec:rls}).
The proofs in this section consolidate the existing literature, and we have attempted to streamline the derivations and obtain sharper constants.
\subsection{Lower bound}
\label{sec:lower_bound}

Here, we prove a lower bound on the number of columns needed for any Nystr\"om method to achieve a $(r, \varepsilon)$-approximation of a worst-case matrix.

\begin{theorem}[Nystr\"om lower bound \cite{DV06,GS12}] \label{thm:lower-bound}
Fix $r \geq 1$ and $\varepsilon > 0$. 
There exists a psd matrix $\mat{A} \in \complex^{N \times N}$ such that
any rank-$k$ Nystr\"om approximation with 
\begin{equation*}
k < r/\varepsilon
\end{equation*}
columns has error 
$\tr \bigl(\mat{A} - \mat{A}^{(k)}\bigr) > (1 + \varepsilon) \cdot \tr \bigl(\mat{A} - \lowrank{\mat{A}}_r\bigr)$.
\end{theorem}

\begin{proof}
We adapt the proof of \cite[Lemma~6.2]{GS12}.
We consider the matrix
\begin{equation*}
\mat{A} = \underbrace{\begin{pmatrix}
    \mat{B} \\
    & \mat{B} \\
    & & \ddots \\
    & & & \mat{B}
    \end{pmatrix}}_{Mr \times Mr}, \quad \text{where} \quad
    \mat{B} = 
    \underbrace{\begin{pmatrix}
    1 & \delta & \cdots & \delta \\
    \delta & 1 & & \vdots \\
    \vdots & & \ddots & \\
    \delta & \cdots & & 1
    \end{pmatrix}}_{r \times r}.
\end{equation*}
We consider the rank-$k$ Nystr\"om approximation 
$\mat{\hat{A}}^{(k)}$
that selects $k_1$ columns from the first block, $k_2$ columns from the second block, etc.
The approximation error satisfies
\begin{equation*}
    \frac{\tr(\mat{A} - \mat{\hat{A}}^{(k)})}
    {\tr (\mat{A} - \lowrank{\mat{A}}_r)}
    = \frac{1}{r} \sum_{i=1}^r \frac{M - k_i}{M - 1} \biggl(1 + \frac{1}{\delta^{-1} + k_i - 1}\biggr)
    \geq \frac{M - k}{M - 1} \cdot \frac{1}{r} \sum_{i=1}^r \biggl(1 + \frac{1}{\delta^{-1} + k_i - 1}\biggr)
\end{equation*}
We use the convexity of $f(x) = \frac{1}{x}$ to calculate
\begin{equation*}
    \frac{1}{r} \sum_{i=1}^r \biggl(1 + \frac{1}{\delta^{-1} + k_i - 1}\biggr)
    \geq 1 + \frac{1}{\delta^{-1} + k \slash r - 1}.
\end{equation*}
Last, we choose $M$ large enough and $\delta$ close enough to $1$ so that
\begin{equation*}
    \frac{\tr(\mat{A} - \mat{\hat{A}}^{(k)})}
    {\tr (\mat{A} - \lowrank{\mat{A}}_r)}
    > 1 + \varepsilon
\end{equation*}
for each $k < \frac{r}{\varepsilon}$.
\end{proof}

\subsection{The greedy method}
\label{sec:greedy}

The \emph{greedy method} \eqref{eq:greedy} is a column Nystr\"om approximation with a long history in numerical analysis \cite{high90c} under the name \textit{complete pivoting} or \textit{diagonal pivoting}.
The papers \cite{BJ05,FS01} popularized the method in the context of kernel computations.
Despite its popularity, however, the greedy method is known to fail when applied to certain input matrices \cite[Ex.~2.1]{high90c},
and the greedy method exhibits poor performance for most of the kernel matrices appearing in
\Cref{sec:compare,sec:krr,sec:spectral_clustering}.
Below in \Cref{thm:greedy}, we construct a worst-case matrix $\mat{A}$ that is approximated at a slow $1 - k/N$ rate using the greedy method.

\begin{theorem}[Greedy method] \label{thm:greedy}
Fix $r \geq 1$ and $\varepsilon > 0$.
Then, the greedy method has the following properties:
\begin{enumerate}[label=(\alph*)]
    \item \label{item:greedy_a} For any psd input matrix $\mat{A} \in \mathbb{C}^{N \times N}$, the greedy method with
    \begin{equation*}
    k \geq (1 - (1 + \varepsilon) \eta) N
    \end{equation*}
    columns produces an approximation satisfying $\expect \tr \bigl(\mat{A} - \mat{A}^{(k)}\bigr) \leq (1 + \varepsilon) \cdot \tr \bigl(\mat{A} - \lowrank{\mat{A}}_r\bigr)$.
    \item \label{item:greedy_b} There exists a psd matrix $\mat{A} \in \mathbb{C}^{N \times N}$ such that the greedy method with
    \begin{equation*}
        k < (1 - (1 + \varepsilon) \eta) N
    \end{equation*}
    columns has error
    $\expect \tr \bigl(\mat{A} - \mat{A}^{(k)}\bigr) > (1 + \varepsilon) \cdot \tr \bigl(\mat{A} - \lowrank{\mat{A}}_r\bigr)$.
\end{enumerate}
As usual, we have defined the relative error $\eta \coloneqq \tr(\mat{A} - \lowrank{\mat{A}}_r) / \tr(\mat{A})$.
\end{theorem}
\begin{proof}
To prove part \ref{item:greedy_a}, observe at each iteration $1 \leq i \leq k$ there are at most $N - i + 1$ nonzero entries in the diagonal of the residual matrix $\mat{A}^{(i - 1)}$,
and the largest entry is incorporated into the pivot set $\set{S}$.
Consequently,
\begin{equation*}
    \tr \mat{A}^{(i)} \leq \Bigl(1 - \frac{1}{N - i + 1}\Bigr) \tr \mat{A}^{(i-1)}
    = \frac{N - i}{N - i + 1} \tr \mat{A}^{(i-1)}.
\end{equation*}
By induction, it follows that $\tr \mat{A}^{(k)} \leq \frac{N - k}{N} \tr \mat{A}$ and
\begin{equation*}
    \frac{\tr \mat{A}^{(k)}}
    {\tr(\mat{A} - \lowrank{\mat{A}}_r)}
    \leq \biggl(1 - \frac{k}{N}\biggr) \, \frac{\tr \mat{A}}{\tr(\mat{A} - \lowrank{\mat{A}}_r)} = \frac{1 - \frac{k}{N}}{\eta}.
\end{equation*}
If $k \geq (1 - (1 + \varepsilon) \eta) N$, the right-hand side is bounded by $1 + \varepsilon$, establishing part \ref{item:greedy_a}.

To prove part \ref{item:greedy_b}, consider the matrix
\begin{equation*}
    \label{eq:worst_case_greedy}
  \mat{A} = 
  \begin{bmatrix}
    \mat{B} & & & \\
    & \mat{C} & & \\
    & & \ddots & \\
    & & & \mat{C}
  \end{bmatrix},
  \quad \text{where} \quad
  \mat{B} = 
  \underbrace{\begin{bmatrix}
    1 & & & \\
    & 1 & & \\
    & & \ddots & \\
    & & & 1
  \end{bmatrix}}_{(N - rM) \times (N - rM)},
  \quad
  \mat{C} = \underbrace{\begin{bmatrix}
      1 & 1 & \cdots & 1 \\
      1 & 1 & \cdots & 1 \\
      \vdots & \vdots & \ddots & \vdots \\
      1 & 1 & \cdots & 1
  \end{bmatrix}}_{M \times M}.
\end{equation*}
In lieu of a good tie-breaking rule, the greedy method chooses entries from the $\mat{B}$ block
before the $\mat{C}$ blocks.
An explicit calculation shows
\begin{equation*}
    \frac{\tr (\mat{A} - \mat{\hat{A}}^{(k)})}{\tr(\mat{A} - \lowrank{\mat{A}}_r)}
    = \frac{N - k}{N - M - r + 1}
    = \frac{1 - \frac{k}{N}}{\eta}
\end{equation*}
for each $k \leq N - M$.
This relative error strictly exceeds $1 + \varepsilon$
as long as $k < (1 - (1 + \varepsilon) \eta) N$, which establishes part \ref{item:greedy_b}.
\end{proof}

\subsection{Uniform sampling}
\label{sec:uniform}

Another popular column Nystr\"om approximation method is \emph{uniform sampling} \cite{WS00,DM05}.
In this method, $k$ columns are selected uniformly at random, either with or without replacement.
Theoretically and empirically, the accuracy is higher using uniform sampling without replacement, which avoids the issue of duplicate column selections \cite{KMT09a}.
Uniform sampling leads to accurate approximations when the dominant $r$ eigenvectors have mass that is spread out equally over all the coordinates (a property known as ``incoherence'') \cite{Git11}.
However, uniform sampling leads to inaccurate approximations when the dominant eigenvectors have mass that is highly concentrated on a subset of the vertices.

Uniform sampling is typically applied to a kernel matrix $\mat{A}$ with ones on the diagonal,
but \emph{diagonal sampling} \cite{FKV04} is a more general method that randomly selects pivots with probabilities proportional to $\diag \mat{A}$.
Diagonal sampling is guaranteed to produce a $(r, \varepsilon)$-approximation when the number of columns satisfies $k \geq (r-1)/(\epsilon \eta) + 1/\varepsilon$, as we will prove in \Cref{thm:diagonal}.
This bound shows that diagonal sampling accurately approximates the dominant rank-one component of a psd matrix, but it can produce inaccurate approximations of the dominant rank-$r$ component for $r > 1$.
A better option for approximating the dominant rank-$r$ component with $r > 1$ is \RPCholesky, which is equivalent to performing diagonal sampling iteratively on the residual matrix.

\begin{theorem}[Diagonal sampling \cite{FKV04}] \label{thm:diagonal}
Fix $r \geq 1$ and $\varepsilon > 0$.
Then, diagonal sampling has the following error properties:
\begin{enumerate}[label=(\alph*)]
    \item \label{item:uniform_a} 
    For any psd input matrix $\mat{A} \in \mathbb{C}^{N \times N}$, diagonal sampling with
    \begin{equation*}
    k \geq \frac{r-1}{\eta} \varepsilon^{-1} + \varepsilon^{-1}
    \end{equation*}
    columns produces an approximation satisfying $\expect \tr \bigl(\mat{A} - \mat{A}^{(k)}\bigr) \leq (1 + \varepsilon) \cdot \tr (\mat{A} - \lowrank{\mat{A}}_r)$.
    \item \label{item:uniform_b} If $r \geq 2$, there exists a psd matrix $\mat{A} \in \mathbb{C}^{N \times N}$
    such that diagonal sampling with
    \begin{equation*}
    k < \frac{r-1}{\eta} (\sqrt{\varepsilon^{-1} + 1} - 1)^2
    \end{equation*}
    columns has error $\expect \tr \bigl(\mat{A} - \mat{A}^{(k)}\bigr) > (1 + \varepsilon) \cdot \tr (\mat{A} - \lowrank{\mat{A}}_r)$.
\end{enumerate}
As usual, we have defined the relative error $\eta \coloneqq \tr(\mat{A} - \lowrank{\mat{A}}_r) / \tr(\mat{A})$.
\end{theorem}
\begin{proof}
Part \ref{item:uniform_a} improves on the earlier error bound \cite[Eq.~(4)]{FKV04}, and it is proved using a more detailed argument with the same technique.
We assume the sampling is conducted with replacement, which leads to higher error.
The trace-norm error takes the form
\begin{equation*}
  \tr \bigl(\mat{A} - \mat{A}(:,\set{S}) \mat{A}(\set{S},\set{S})^\dagger \mat{A}(\set{S},:)\bigr)
  = \tr \bigl(\mat{A}^{1 \slash 2}
  \bigl(\Id - \mat{\Pi}_{\mat{A}^{1 \slash 2}(:,\set{S})}\bigr)
  \mat{A}^{1 \slash 2}\bigr)
  = \lVert \mat{A}^{1 \slash 2}
  \bigl(\Id - \mat{\Pi}_{\mat{A}^{1 \slash 2}(:,\set{S})}\bigr) \rVert^2_{\rm F},
\end{equation*}
where $\mat{\Pi}_{\mat{A}^{1 \slash 2}(:,\set{S})}$ is the orthogonal projector onto the range of $\mat{A}^{1 \slash 2}(:,\set{S})$.
Since $\mat{A}^{1 \slash 2} \mat{\Pi}_{\mat{A}^{1 \slash 2}(:,\set{S})}$ is the optimal Frobenius norm approximation of $\mat{A}^{1 \slash 2}$ in the row space of $\mat{A}^{1 \slash 2}(\set{S}, :)$, we observe the inequality
\newcommand{\ieigenvec}{{\vec{v}_i}}
\begin{equation*}
  \bigl\lVert \mat{A}^{1 \slash 2}
  \bigl(\Id - \mat{\Pi}_{\mat{A}^{1 \slash 2}(:,\set{S})}\bigr) \bigr\rVert^2_{\rm F}
  \leq \Biggl\lVert \mat{A}^{1 \slash 2}
  - \Biggl(\sum_{j=1}^r \vec{v}_j \vec{v}_j^* \Biggr) \Biggl(\frac{\tr \mat{A}}{k} \sum_{j=1}^k \frac{\mathbf{e}_{s_j} \mathbf{e}_{s_j}^*}{\mat{A}(s_j, s_j)} \Biggr) \mat{A}^{1 \slash 2} \Biggr\rVert^2_{\rm F},
\end{equation*}
where $\mathbf{e}_{s_i}$ is the unit vector in the direction of the $i$th random pivot and $\ieigenvec$ denotes the $i$th eigenvector of $\mat{A}$.
Taking transposes and using the orthonormal basis of eigenvectors $\vec{v}_1, \ldots, \vec{v}_N$, we calculate
\begin{align*}
  & \Biggl\lVert \mat{A}^{1 \slash 2}
  - \Biggl(\sum_{j=1}^r \vec{v}_j \vec{v}_j^* \Biggr) \Biggl(\frac{\tr \mat{A}}{k} \sum_{j=1}^k \frac{\mathbf{e}_{s_j} \mathbf{e}_{s_j}^*}{\mat{A}(s_j, s_j)} \Biggr) \mat{A}^{1 \slash 2} \Biggr\rVert^2_{\rm F} \\
  &= \sum_{i=1}^r \Biggl\lVert \mat{A}^{1 \slash 2} \biggl(\Id
  - \frac{\tr \mat{A}}{k} \sum_{j=1}^k \frac{\mathbf{e}_{s_j} \mathbf{e}_{s_j}^*}{\mat{A}(s_j, s_j)} \Biggr) \vec{v}_i \Biggr\rVert^2
  + \sum_{i=r+1}^N \bigl\lVert \mat{A}^{1 \slash 2} \vec{v}_i \bigr\rVert^2
\end{align*}
Using the characterization of $\tr\bigl(\mat{A} - \lowrank{\mat{A}}_r\bigr)$ in terms of the eigenvalues $\lambda_1(\mat{A}) \geq \cdots \geq \lambda_N(\mat{A})$, we find
\begin{equation*}
    \sum_{i=r+1}^N \bigl\lVert \mat{A}^{1 \slash 2} \vec{v}_i \bigr\rVert^2 
    = \sum_{i=r+1}^N \lambda_i(\mat{A})
    = \tr\bigl(\mat{A} - \lowrank{\mat{A}}_r\bigr)
\end{equation*}
Next, observe that the random vectors
\begin{equation*}
    \mat{A}^{1 \slash 2} \Biggl(\tr \mat{A}
    \frac{\mathbf{e}_{s_j} \mathbf{e}_{s_j}^*}{\mat{A}(s_j, s_j)} \vec{v}_i \Biggr)
\end{equation*}
are independent for $i = 1, \ldots, k$,
and each vector has mean $\mat{A}^{1 \slash 2} \vec{v}_i$
and expected square norm $\tr \mat{A}$.
This allows us to calculate
\begin{equation*}
  \expect \Biggl\lVert \mat{A}^{1 \slash 2}
  \Biggl(\Id - \frac{\tr \mat{A}}{k} \sum_{j=1}^k \frac{\mathbf{e}_{s_j} \mathbf{e}_{s_j}^*}{\mat{A}(s_j, s_j)}\Biggr) \vec{v}_i \Biggr\rVert^2
  = \frac{\tr \mat{A} - \lambda_i(\mat{A})}{k}.
\end{equation*}
Summing over $i = 1, \ldots, r$ guarantees the error bound
\begin{equation*}
  \expect \tr\bigl(\mat{A} - \mat{\hat{A}^{(k)}}\bigr) \leq \frac{r - 1}{k} \tr \mat{A} + \Bigl(1 + \frac{1}{k}\Bigr) \tr(\mat{A} - \lowrank{\mat{A}}_r),
\end{equation*}
which completes part \ref{item:uniform_a} of the theorem.

To prove part \ref{item:uniform_b}, we assume the sampling is conducted without replacement and consider the $N \times N$ matrix
\begin{equation*}
\mat{A} = \begin{bmatrix}
    \mat{B} \\
    & \ddots \\
    & & \mat{B} \\
    & & & \mat{C}
    \end{bmatrix},
    \quad
    \text{where}
    \quad
    \mat{B} = \underbrace{\begin{bmatrix}
        1 & \delta & \cdots & \delta \\
        \delta & 1 & \cdots & \delta \\
        \vdots & \vdots & \ddots & \vdots \\
        \delta & \delta & \cdots & 1
    \end{bmatrix}}_{M \times M},
    \quad
    \mat{C} = \underbrace{\begin{bmatrix}
        1 & 1 & \cdots & 1 \\
        1 & 1 & \cdots & 1 \\
        \vdots & \vdots & \ddots & \vdots \\
        1 & 1 & \cdots & 1
    \end{bmatrix}}_{N - M(r-1) \times N - M(r-1)}.
\end{equation*}
Next, consider the rank-$k$ Nystr\"om approximation 
$\mat{\hat{A}}^{(k)}$
that selects $k_1$ columns from the first block, $k_2$ columns from the second block, etc.
The Schur complement of $\mat{B}$ with respect to any $k_i$ distinct columns has 
trace 
\begin{equation*}
    (M - k_i) \biggl(1 + \frac{1}{\delta^{-1} + k_i  - 1}\biggr) (1 - \delta).
\end{equation*}
Using the fact $\tr (\mat{A} - \lowrank{\mat{A}}_r) = (1 - \delta)(r - 1)(M - 1)$, we calculate
\begin{equation*}
\frac{\tr(\mat{A} - \mat{\hat{A}}^{(k)})}{\tr (\mat{A} - \lowrank{\mat{A}}_r)}
\geq \frac{1}{r-1} \sum_{i=1}^{r-1} \frac{M - k_i}{M - 1} \biggl(1 + \frac{1}{\delta^{-1} + k_i - 1}\biggr).
\end{equation*}
We use the convexity of $f(x) = \frac{1}{x}$
and the fact that
$\expect k_i 
= \frac{M k}{N}$
for $1 \leq i \leq r - 1$,
calculate
\begin{equation*}
    \frac{\expect \tr(\mat{A} - \mat{\hat{A}}^{(k)})}{\tr (\mat{A} - \lowrank{\mat{A}}_r)}
    \geq
    \frac{M - k}{M - 1}
    \biggl(1 + \frac{1}{
    \delta^{-1} + \frac{Mk}{N} - 1}\biggr).
\end{equation*}
The worst case occurs when we take $\delta = \sqrt{\frac{\varepsilon}{\varepsilon + 1}}$ and let the dimensions $M$ and $N$ grow to infinity, with fixed aspect ratio.
Then, we use the identity $\eta N = (r-1)(M - 1)(1 - \delta)$ to show that
the right-hand side converges
\begin{equation*}
    \frac{M - k}{M - 1}
    \biggl(1 + \frac{1}{
    \delta^{-1} + \frac{Mk}{N} - 1}\biggr)
    \rightarrow 1 + \frac{1}{\delta^{-1} + \frac{k \eta}{(r - 1)(1 - \delta)} - 1}
\end{equation*}
To make this quantity smaller than $1 + \varepsilon$,
uniform sampling requires at least
\begin{equation*}
    k \geq \frac{(r - 1) (1 - \delta)}{\eta} \bigl[ \varepsilon^{-1} + 1 - \delta^{-1} \bigr]
    = \frac{r - 1}{\eta} (\sqrt{\varepsilon^{-1} + 1} - 1)^2
\end{equation*}
columns.
This completes the proof of part \ref{item:uniform_b}.
\end{proof}

\subsection{Determinantal point process sampling}
\label{sec:dpp}

Determinantal point process (DPP) sampling \cite{DM21a} is a column Nystr\"om approximation method that selects a pivot set of cardinality $|\set{S}| = k$ according to the distribution
\begin{equation}
\label{eq:exact_dpp}
\prob \bigl\{ \set{S} = \{s_1, \ldots, s_k\} \bigr\} = \frac{\det \mat{A}(\set{S}, \set{S})}
{\sum_{|\set{S}^{\prime}| = k}
\det \mat{A}(\set{S}^{\prime}, \set{S}^{\prime})}.
\end{equation}
DPP sampling has nearly optimal $(r, \varepsilon)$-approximation properties, as we will show in \Cref{thm:near-optimal}.
However, implementing DPP sampling for large $k$ values remains expensive relative to peer methods \cite{AGR16,DCV19}.

In the DPP sampling literature, there is a surprising connection between \RPCholesky and $k$-DPP sampling:
when we apply $k$ steps of \RPCholesky to a rank-$k$ orthogonal projection matrix,
we obtain \emph{exactly} the same distribution as $k$-DPP sampling
\cite{gillenwater2014thesis,Pou20}.
This observation leads to one of the standard strategies for $k$-DPP sampling, based on a reduction to rank-$k$ orthogonal projection matrices \cite{kulesza2012determinantal}:
\begin{enumerate}
    \item Calculate the full eigendecomposition of the target matrix.
    \item Randomly select a set of $k$ eigenvectors with probability proportional to the product of the $k$ associated eigenvalues.
    \item Form the orthogonal projection matrix using the $k$ eigenvectors.
    \item Apply \RPCholesky to the projection matrix to obtain the set $\set{S}$.
\end{enumerate}
In step 1, the full eigendecomposition requires $\mathcal{O}(N^3)$ operations.
It is much cheaper to apply \RPCholesky directly,
which is equivalent to performing $1$-DPP sampling iteratively on the residual matrix.

\begin{theorem}[$k$-DPP sampling \cite{BW09,GS12}] \label{thm:near-optimal}
Fix $r \geq 1$ and $\varepsilon > 0$.
Then, the Nystr\"om approximation produced by $k$-DPP sampling has the following properties:
\begin{enumerate}[label=(\alph*)]
    \item \label{item:dpp_a} For any psd input matrix $\mat{A} \in \mathbb{C}^{N \times N}$, $k$-DPP sampling with
    \begin{equation*}
    k \geq r/\varepsilon + r - 1
    \end{equation*}
    columns produces an approximation satisfying $\expect \tr \bigl(\mat{A} - \mat{A}^{(k)}\bigr) \leq (1 + \varepsilon) \cdot \tr \bigl(\mat{A} - \lowrank{\mat{A}}_r\bigr)$.
    \item \label{item:dpp_b} There exists a psd matrix $\mat{A} \in \mathbb{C}^{N \times N}$ such that
    $k$-DPP sampling with 
    \begin{equation*}
    k < r/\varepsilon + r-1
    \end{equation*}
    columns leads to error $\expect \tr \bigl(\mat{A} - \mat{A}^{(k)}\bigr) > (1 + \varepsilon) \cdot \tr \bigl(\mat{A} - \lowrank{\mat{A}}_r\bigr)$.
\end{enumerate}
\end{theorem}
\begin{proof}
These results were essentially proved in \cite[Thm.~1]{BW09} and \cite{GS12}, but for completeness we provide a streamlined derivation here.
Let $\mat{A} \slash \mat{A}(\set{S}, \set{S})$ denote the Schur complement of $\mat{A}$ with respect to the coordinates $\set{S} = \{s_1, \ldots, s_k\}$.
Recall the Crabtree--Haynsworth determinant identity \cite[Lem.~1]{CH69}:
\begin{equation}
\label{eq:ch}
    (\mat{A} \slash \mat{A}(\set{S}, \set{S}))(i,j)
    = \frac{\det \mat{A}(\set{S} \cup \{i\}, \set{S} \cup \{j\})}
    {\det \mat{A}(\set{S}, \set{S})}
\end{equation}
for $i, j \notin \set{S}$.
Also recall the determinant identity \cite[Lem.~2.1]{GS12}:
\begin{equation}
\label{eq:sym}
    \sum_{|\set{S}| = k}
    \det \mat{A}(\set{S}, \set{S}) = e_k(\lambda_1(\mat{A}), \ldots, \lambda_N(\mat{A})),
\end{equation}
where
\begin{equation*}
    e_k(\lambda_1(\mat{A}), \ldots, \lambda_N(\mat{A}))
    = \sum_{|\set{S}| = k} \prod_{i \in \set{S}} \lambda_i(\mat{A})
\end{equation*}
is the $k$th elementary symmetric polynomial evaluated on the eigenvalues of $\mat{A}$.
Using \eqref{eq:ch} and \eqref{eq:sym}, we can calculate the
error of $k$-DPP sampling exactly: 
\begin{align*}
    \expect \tr(\mat{A} - \mat{\hat{A}}^{(k)}) 
    &= \frac{\sum_{|\set{S}| = k} \det \mat{A}(\set{S}, \set{S})
    \tr(\mat{A} \slash \mat{A}(\set{S}, \set{S}))}
    {\sum_{|\set{S}^{\prime}| = k}
    \det \mat{A}(\set{S}^{\prime}, \set{S}^{\prime})}
    = \frac{\sum_{|\set{S}| = k} \sum_{i \notin \set{S}} 
    \det \mat{A}(\set{S} \cup \{i\}, \set{S} \cup\{i\})}
    {\sum_{|\set{S}^{\prime}| = k}
    \det \mat{A}(\set{S}^{\prime}, \set{S}^{\prime})} \\
    &= 
    (k+1) \frac{\sum_{|\set{S}| = k+1}
    \det \mat{A}(\set{S}, \set{S})}
    {\sum_{|\set{S}^{\prime}| = k}
    \det \mat{A}(\set{S}^{\prime}, \set{S}^{\prime})} =
    (k+1) \frac{e_{k+1}(\lambda_1(\mat{A}), \ldots, \lambda_N(\mat{A}))}
    {e_k(\lambda_1(\mat{A}), \ldots, \lambda_N(\mat{A}))},
\end{align*}
Remarkably, this error is the same for a diagonal and non-diagonal matrix, so we might as well assume $\mat{A}$ is diagonal.
Next, as noted by \cite{GS12}, the function
\begin{equation*}
    f(x_1, \ldots, x_N)
    = \frac{e_{k+1}(x_1, \ldots, x_N)}
    {e_k(x_1, \ldots, x_N)}
\end{equation*}
is concave, non-decreasing in all of its arguments,
and invariant under permutations of its arguments.
Therefore, averaging together some of the arguments cannot decrease the value of $f$.
For every $(x_1, \ldots, x_N)$, it follows that
\begin{equation*}
    f(x_1, \ldots, x_N) 
    \leq f\Biggl(\underbrace{\sum_{i=1}^r \frac{x_i}{r}, \ldots, \sum_{i=1}^r \frac{x_i}{r}}_{r\,\text{times}},
    \underbrace{\sum_{i=r+1}^N \frac{x_i}{N-r}, \ldots, \sum_{i=r+1}^N \frac{x_i}{N-r}}_{N-r\,\text{times}}\Biggr)
\end{equation*}
Additionally, for every $(a, r, b, N)$,
\begin{align*}
    f\Bigl(\underbrace{\frac{a}{r}, \ldots, \frac{a}{r}}_{r\,\text{times}},
    \underbrace{\frac{b}{N-r}, \ldots, \frac{b}{N-r}}_{N-r\,\text{times}}\Bigr)
    &= f\Bigl(\underbrace{\frac{a}{r}, \ldots, \frac{a}{r}}_{r\,\text{times}},
    \underbrace{\frac{b}{N-r}, \ldots, \frac{b}{N-r}}_{N-r\,\text{times}}, 0\Bigr) \\
    &\leq f\Bigl(\underbrace{\frac{a}{r}, \ldots, \frac{a}{r}}_{r\,\text{times}},
    \underbrace{\frac{b}{N-r+1}, \ldots, \frac{b}{N-r+1}}_{N-r+1\,\text{times}}\Bigr).
\end{align*}
Consequently, $k$-DPP sampling achieves the worst-case error for the diagonal matrix
\begin{equation}
\label{eq:worst_case}
    \mat{A} = \diag\Bigl(\underbrace{\frac{a}{r}, \ldots, \frac{a}{r}}_{r\,\text{times}},
    \underbrace{\frac{b}{N-r}, \ldots, \frac{b}{N-r}}_{N-r\,\text{times}}\Bigr)
\end{equation}
in the limit as $a \rightarrow \infty$ and $N \rightarrow \infty$.
Assuming diagonal $\mat{A}$ and $k \geq r$, the dominant error arises when $k$-DPP sampling selects just $r - 1$ of the $r$ large diagonal entries.
The error is explicitly
\begin{align*}
    \lim_{N,a \rightarrow \infty} f\Bigl(\underbrace{\frac{a}{r}, \ldots, \frac{a}{r}}_{r\,\text{times}},
    \underbrace{\frac{b}{N-r}, \ldots, \frac{b}{N-r}}_{N-r\,\text{times}}\Bigr)
    &= \lim_{N,a \rightarrow \infty}
    \Biggl[
    b + \frac{a}{r} \cdot
    \frac{\binom{r}{r - 1} \binom{N - r}{k - r + 1} \bigl(\frac{a}{r}\bigr)^{r-1} \bigl(\frac{b}{N-r}\bigr)^{k-r+1}}
    {\binom{r}{r} \binom{N - r}{k - r} \bigl(\frac{a}{r}\bigr)^r \bigl(\frac{b}{N-r}\bigr)^{k-r}} \Biggr] \\
    &= \Bigl(1 + \frac{r}
    {k-r+1}\Bigr) b,
\end{align*}
whence
\begin{equation*}
    \expect \tr(\mat{A} - \mat{\hat{A}}^{(k)}) \leq \Bigl(1 + \frac{r}
    {k-r+1}\Bigr) b.
\end{equation*}
Because this is an explicit expression for the worst-case error, setting $k \geq \frac{r}{\varepsilon} + r - 1$ always guarantees an $(r, \varepsilon)$-approximation.
Conversely, when $k < \frac{r}{\varepsilon} + r - 1$,
$k$-DPP sampling fails to produce an $(r, \varepsilon)$-approximation for a diagonal matrix of the form \eqref{eq:worst_case}
with $N$ and $a$ chosen sufficiently high.
\end{proof}

\subsection{Ridge leverage score sampling}
\label{sec:rls}

Ridge leverage score (RLS) sampling \cite{AM15,CLV17,MM17,RCCR18} is a Nystr\"om approximation with a sampling distribution that is potentially more tractable than in DPP sampling. 
To perform RLS sampling with parameter $\lambda > 0$, we first introduce the vector of ridge leverage scores:
\begin{equation} \label{eq:rls}
    \vec{\ell}^{\lambda}
    = \diag (\mat{A} (\mat{A} + \lambda \Id)^{-1}).
\end{equation}
Then, we calculate the vector of sampling probabilities \cite{MM17}:
\begin{equation*}
    \vec{p} = \min\bigl\{1, f \cdot \vec{\ell}^{\lambda} \bigr\},
\end{equation*}
where $f > 1$ is the oversampling factor.
Last, we generate the coordinate set $\set{S} \subseteq \{1,\ldots,N\}$
by independently including each index $i$ with probability $\vec{p}(i)$.
The available implementations of RLS sampling \cite{CLV17,MM17,RCCR18} are all fairly complicated,
as they require selecting parameters $\lambda$ and $f$ and approximating the resulting RLS sampling distribution.
Because of these preprocessing steps, RLS sampling requires a significantly higher number of entry evaluations than \RPCholesky for a fixed approximation rank $k$.
In practice, we have found that RLS is also less reliable.

In 2017, Musco \& Musco analyzed RLS sampling and proved it produces good low-rank approximations at moderate cost.
Here is a slightly 
simplified version of one of their results \cite[Thm.~18]{MM17}:
\begin{theorem}[Ridge leverage score sampling: probability bound \cite{MM17}]
    For any psd matrix $\mat{A}$, there exist parameters $\lambda,f>0$ such that RLS sampling produces a column Nystr\"om approximation $\mat{\hat{A}}$ such that
    \begin{equation} \label{eq:probability_bound}
        \tr (\mat{A} - \mat{\hat{A}}) \le (1+\varepsilon) \tr(\mat{A} - \lowrank{\mat{A}}_r) \quad \text{with probability at least }1-\delta.
    \end{equation}
    The matrix $\mat{\hat{A}}$ has rank
    \begin{equation} \label{eq:musco_rank}
        k = \order\left( \frac{r}{\varepsilon} \log \left(\frac{r}{\delta\varepsilon}\right) \right).
    \end{equation}
\end{theorem}
This result is not directly comparable to our $(r,\varepsilon)$-approximation guarantees for \RPCholesky because the error bound \eqref{eq:probability_bound} controls the trace-norm error up to a failure probability rather than in expectation.
In addition, the Musco--Musco result does not have explicit constants for the approximation rank $k$.

To obtain results for RLS sampling that are directly comparable to our own, we reanalyzed RLS sampling, resulting in \Cref{thm:musco} below.
Our proof, following \cite[Thm.~3]{MM17}, uses the matrix Bernstein inequality to show that $\mathbb{P}\{\lVert \mat{A} - \mat{\hat{A}} \rVert > \lambda \}$ decreases exponentially fast as we increase the oversampling $f$.
By appropriately choosing $\lambda$ and $f$, we are able to guarantee an $(r, \varepsilon)$-approximation.
Our resulting error bounds for RLS sampling in \Cref{thm:musco} depend primarily on $r + r/\varepsilon$, similar to the bounds for DPP sampling.
In order to pass from the probability bound \eqref{eq:probability_bound} to an expectation bound \eqref{eq:1+eps}, the approximation rank $k$ acquires a logarithmic dependence on the inverse relative error $1/\eta$. 

\begin{theorem}[Ridge leverage score sampling: expectation bound, explicit constants]
\label{thm:musco}
Fix $r \geq 1$ and $\varepsilon > 0$.
For any psd input matrix $\mat{A} \in \mathbb{C}^{N \times N}$, the approximation $\mat{\hat{A}}$ produced by RLS sampling with parameters 
\begin{equation*}
    \lambda = \frac{\varepsilon}{2r} \tr (\mat{A} - \lowrank{\mat{A}}_r)
    \quad \text{and} \quad
    f = 27 \log\Bigl(\frac{4}{\eta} \Bigl(r + \frac{r}{\varepsilon} \Bigr) \Bigr)
\end{equation*}
has the following error properties:
\begin{enumerate}[label=(\alph*)]
\item \label{item:rls_a} With probability at least $1 - \varepsilon \eta / 2$,
the Nystr\"om approximation $\mat{\hat{A}}$ satisfies
\begin{equation*}
    \frac{\tr \bigl(\mat{A} - \mat{\hat{A}} \bigr)}
    {\tr (\mat{A} - \lowrank{\mat{A}}_r)} \leq 1 + \frac{\varepsilon}{2}.
    \qquad
    \rank \mat{\hat{A}} \leq 65 \Bigl(r + \frac{r}{\varepsilon}\Bigr) \log\Bigl(\frac{4}{\eta} \Bigl(r + \frac{r}{\varepsilon} \Bigr) \Bigr).
\end{equation*}
\item \label{item:rls_b} Define the truncation rank
\begin{equation*}
    k = 65 \Bigl(r + \frac{r}{\varepsilon}\Bigr) \log\Bigl(\frac{4}{\eta} \Bigl(r + \frac{r}{\varepsilon} \Bigr) \Bigr),
\end{equation*}
and set $\mat{\hat{A}}^{(k)} = \mat{\hat{A}}$ if $\rank \mat{\hat{A}} \leq k$ and $\mat{\hat{A}}^{(k)} = \mat{0}$ otherwise.
Then, $\mat{\hat{A}}^{(k)}$ is a Nystr\"om approximation with rank at most $k$, which satisfies $\expect \tr \bigl(\mat{A} - \mat{A}^{(k)}\bigr) \leq (1 + \varepsilon) \cdot \tr \bigl(\mat{A} - \lowrank{\mat{A}}_r\bigr)$.
\end{enumerate}
As usual, we have defined the relative error $\eta \coloneqq \tr(\mat{A} - \lowrank{\mat{A}}_r) / \tr(\mat{A})$.
\end{theorem}
\begin{proof}
    To prove part \ref{item:rls_a}, we start by bounding the rank of the Nystr\"om approximation.
    We observe $\rank \mat{\hat{A}} \leq |\set{S}|$,
    where $\set{S}$ denotes the set of indices sampled using RLS sampling.
    From the description of RLS sampling, $|\set{S}|$ is the sum of independent Bernoulli random variables, with expected value
    \begin{equation*}
        \mathbb{E} |\set{S}| = \sum_{i=1}^N \vec{p}(i) \leq f \sum_{i=1}^N \vec{\ell}^\lambda(i).
    \end{equation*}
    The sum of the leverage scores $\sum_{i=1}^N \ell_i^\lambda$ is bounded by
    \begin{equation}
    \label{eq:leverage}
        \sum_{i=1}^N \vec{\ell}^\lambda(i) = \tr\bigl(\mat{A}\bigl(\mat{A} + \lambda \Id\bigr)^{-1}\bigr)
        = \sum_{i=1}^N \frac{\lambda_i(\mat{A})}{\lambda + \lambda_i(\mat{A})}
        \leq r + \frac{1}{\lambda} \sum_{i>r} \lambda_i(\mat{A}) 
        \leq 2\Bigl(r + \frac{r}{\varepsilon}\Bigr),
    \end{equation}
    where we have substituted $\lambda = \varepsilon \sum_{i > r} \lambda_i(\mat{A}) / (2r)$.
    It follows that $\mathbb{E} |\set{S}| \leq t$, where
    \begin{equation*}
        t = 54 \Bigl(r + \frac{r}{\varepsilon}\Bigr) \log\Bigl(\frac{4}{\eta} \Bigl(r + \frac{r}{\varepsilon} \Bigr) \Bigr).
    \end{equation*}
    We apply Chernoff's inequality \cite[Thm.~2.3.1]{vershynin2018high} with $\delta \approx 0.199$
    to yield
    \begin{equation*}
        \mathbb{P}\Bigl\{|\set{S}| > (1 + \delta) t \Bigr\} 
        < {\rm e}^{-\mathbb{E} |\set{S}|}
        \Biggl(\frac{{\rm e}\, \mathbb{E} |\set{S}|}{(1 + \delta) t}\Biggr)^{(1 + \delta) t}
        \leq {\rm e}^{-t}
        \Biggl(\frac{{\rm e}\, t}{(1 + \delta) t}\Biggr)^{(1 + \delta) t}
        = {\rm e}^{-t / 54}
        \leq \frac{\varepsilon \eta}{4}.
    \end{equation*}
    With probability at least $1 - \varepsilon \eta / 4$,
    we have shown
    \begin{equation*}
        \rank \mat{\hat{A}} \leq (1 + \delta) t \leq 65 \Bigl(r + \frac{r}{\varepsilon}\Bigr) \log\Bigl(\frac{4}{\eta} \Bigl(r + \frac{r}{\varepsilon} \Bigr) \Bigr).
    \end{equation*}
    
    Next, we bound the spectral norm approximation error $\lVert \mat{A} - \mat{\hat{A}} \rVert$.
    To that end, consider the random rank-one matrices
    \begin{equation*}
        \mat{X}_i = 
        \begin{cases}
            \bigl(\frac{1}{\vec{p}(i)} - 1\bigr) \mat{B}(:, i) \mat{B}(i, :) & i \in \set{S}, \\
            - \mat{B}(:, i) \mat{B}(i, :), & i \notin \set{S},
        \end{cases}
        \quad \text{where} \quad
        \mat{B} = \mat{A}^{1/2} (\mat{A} + \lambda \Id)^{-1/2}.
    \end{equation*}
    Each matrix $\mat{X}_i$ is mean-zero for $1 \leq i \leq N$.
    The matrix $\mat{X}_i$ is exactly zero if the $i$th leverage score $\ell_i^\lambda$ is as large or larger than $1/f$.
    Otherwise, the matrix $\mat{X}_i$ is bounded from above by
    \begin{equation*}
        \lambda_{\max}(\mat{X}_i) \leq \frac{1}{\vec{p}(i)} \lVert \mat{B}(:, i) \mat{B}(i, :) \rVert = \frac{1}{\vec{p}(i)} \vec{\ell}^\lambda(i) = \frac{1}{f}.
    \end{equation*}
    and bounded from below by
    \begin{equation*}
        \lambda_{\min}(\mat{X}_i) \geq -\lVert \mat{B}(:, i) \mat{B}(i, :) \rVert = - \vec{\ell}^\lambda(i) \geq -\frac{1}{f}.
    \end{equation*}
    Hence, $\lVert \mat{X_i} \rVert \leq 1 / f$.
    We upper bound the variance of $\sum_{i=1}^N \mat{X}_i$ as
    \begin{align*}
        \sum_{i=1}^N \mathbb{E} \mat{X}_i^2
        &= \sum_{\vec{\ell}^\lambda(i) < 1/f} \Biggl(\frac{1}{\vec{p}(i)} - 1\Biggr) \mat{B}(:, i) \mat{B}(i, :) \mat{B}(:, i) \mat{B}(i, :) \\
        &\preceq \sum_{\vec{\ell}^\lambda(i) < 1/f} \frac{1}{\vec{p}(i)} \vec{\ell}^\lambda(i) \mat{B}(:, i) \mat{B}(i, :) 
        \preceq \frac{1}{f} \mat{A} (\mat{A} + \lambda \Id)^{-1}
    \end{align*}
    By increasing the largest eigenvalue of the right-hand side to $1/f$,
    we obtain a right-hand side matrix with spectral norm $1/f$ and trace at most $\frac{1}{f} \bigl(\sum_{i=1}^N \vec{\ell}^\lambda(i) + 1\bigr)$.
    Therefore, the matrix Bernstein inequality \cite[Thm.~7.7.1]{tropp2015introduction} gives
    \begin{align*}
        \mathbb{P}\Biggl\{\lambda_{\min}\Biggl(\sum_{i=1}^N \mat{X}_i \Biggr) < -\frac{1}{2} \Biggr\}
        &\leq 4 \Biggl(\sum_{i=1}^N \vec{\ell}^\lambda(i) + 1\Biggr) \exp\biggl(-\frac{3}{28} f \biggr) < \frac{\varepsilon \eta}{4}.
    \end{align*}
    Here, we have used the fact that
    \begin{equation*}
        \frac{28}{3} \log\Biggl(\frac{16}{\varepsilon \eta} \Biggl(\sum_{i=1}^N \vec{\ell}^\lambda(i) + 1\Biggr)\Biggr) \leq f = 27 \log\Biggl(\frac{4}{\eta} \Biggl(r + \frac{r}{\varepsilon} \Biggr) \Biggr),
    \end{equation*}
    which can be shown by a direct calculation using \eqref{eq:leverage}.
    With probability at least $1 - \varepsilon \eta / 4$, we have shown that $-\frac{1}{2} \Id \preceq \sum_{i=1}^N \mat{X}_i$. By multiplying both sides with $(\mat{A} + \lambda \Id)^{1/2}$, we obtain
    \begin{equation*}
        -\frac{1}{2} (\mat{A} + \lambda \Id)
        \preceq
        (\mat{A} + \lambda \Id)^{1/2} \Biggl(\sum_{i=1}^N \mat{X}_i\Biggr) (\mat{A} + \lambda \Id)^{1/2}.
    \end{equation*}
    Using the definition of $\mat{X}_i$, the right-hand side is exactly $\sum_{i \in \set{S}} \frac{1}{\vec{p}(i)} \mat{A}^{1/2}(:, i) \mat{A}^{1/2}(i, :) - \mat{A}$.
    Hence, we can simplify the expression to yield
    \begin{equation*}
        \mat{A}
        \preceq
        2 \sum_{i \in \set{S}} \frac{1}{\vec{p}(i)} \mat{A}^{1/2}(:, i) \mat{A}^{1/2}(i, :)
        + \lambda \Id.
    \end{equation*}
    By multiplying both sides with the orthogonal projection $\Id - \mat{\Pi}_{\mat{A}^{1/2}(:, \set{S})}$, we obtain
    \begin{equation*}
        (\Id - \mat{\Pi}_{\mat{A}^{1/2}(:, \set{S})}) \mat{A} (\Id - \mat{\Pi}_{\mat{A}^{1/2}(:, \set{S})})
        \preceq \lambda (\Id - \mat{\Pi}_{\mat{A}^{1/2}(:, \set{S})}).
    \end{equation*}
    The right-hand side is bounded from above by $\lambda \mat{I}$,
    so we have shown
    $\lVert (\Id - \mat{\Pi}_{\mat{A}^{1/2}(:, \set{S})}) \mat{A} (\Id - \mat{\Pi}_{\mat{A}^{1/2}(:, \set{S})}) \rVert \leq \lambda$.
    By considering the singular value decomposition for $(\Id - \mat{\Pi}_{\mat{A}^{1/2}(:, \set{S})}) \mat{A}^{1/2}$, we arrive at the spectral norm error bound
    \begin{equation*}
        \lVert \mat{A} - \mat{\hat{A}} \rVert
        = \lVert \mat{A}^{1/2} (\Id - \mat{\Pi}_{\mat{A}^{1/2}(:, \set{S})}) \mat{A}^{1/2} \rVert
        = \lVert (\Id - \mat{\Pi}_{\mat{A}^{1/2}(:, \set{S})}) \mat{A} (\Id - \mat{\Pi}_{\mat{A}^{1/2}(:, \set{S})}) \rVert \leq \lambda.
    \end{equation*}
    We can convert the spectral-norm error bound into a trace-norm error bound by calculating
    \begin{equation*}
        \tr (\mat{A} - \mat{\hat{A}}) 
        = \sum_{i=1}^N \lambda_i (\mat{A} - \mat{\hat{A}}) 
        \leq r \lambda + \sum_{i=r + 1}^N \lambda_i (\mat{A} )
        = \Bigl(1 + \frac{\varepsilon}{2}\Bigr) \cdot \tr (\mat{A} - \lowrank{\mat{A}}_r).
    \end{equation*}
    Here, we have used the Weyl monotonicity principle \cite[Thm.~8.11]{Zha11}, which guarantees that $\lambda_i(\mat{A} - \mat{\hat{A}}) \leq \lambda_i(\mat{A})$ for $1 \leq i \leq N$ since $\mat{\hat{A}} \succeq \mat{0}$.
    This completes the proof of part \ref{item:rls_a}.

    To prove part \ref{item:rls_b}, we consider the failure event which occurs when the rank or spectral norm approximation error exceeds the bounds in part \ref{item:rls_b}. Even on the failure event, the trace-norm error is bounded by
    \begin{equation*}
    \tr (\mat{A} - \mat{\hat{A}}^{(k)}) \leq \tr \mat{A}.
    \end{equation*}
    Since the failure event can only occur with probability $(\varepsilon \eta)/2$,
    we arrive at the conclusion
    \begin{align*}
        \expect \tr \bigl(\mat{A} - \mat{A}^{(k)}\bigr)
        &= \expect \bigl[ \tr \bigl(\mat{A} - \mat{A}^{(k)}\bigr) \mathds{1}\{\rm success\}\bigr]
        + \expect \bigl[ \tr \bigl(\mat{A} - \mat{A}^{(k)}\bigr) \mathds{1}\{\rm failure\}\bigr] \\
        &\leq \Bigl(1 + \frac{\varepsilon}{2}\Bigr) \cdot \tr (\mat{A} - \lowrank{\mat{A}}_r)
        + \frac{\varepsilon}{2} \cdot \tr (\mat{A} - \lowrank{\mat{A}}_r)
    \end{align*}
    This completes the proof of part \ref{item:rls_b}.
\end{proof}

\bibliographystyle{abbrvnat}
{\small
\bibliography{refs}}
\newpage
\appendix

\end{document}